\documentclass{article}
\usepackage{amsfonts}
\usepackage{amsmath,amsthm}

\usepackage[all]{xy}
\usepackage{graphicx}

\usepackage{amssymb}
\usepackage{latexsym}

\DeclareMathAlphabet\EuFrak{U}{euf}{m}{n}	
\SetMathAlphabet\EuFrak{bold}{U}{euf}{b}{n}	

\newcommand{\Ra}{\Rightarrow}

\newcommand{\hra}{\hookrightarrow}

\newcommand{\ovl}{\overline}
\newcommand{\unl}{\underline}
\newcommand{\wa}{\widehat}
\newcommand{\wt}{\widetilde}

\newcommand{\sC}{{\it C*}-}

\newcommand{\bC} {{\mathbb C}}

\newcommand{\bR} {{\mathbb R}}
\newcommand{\bT} {{\mathbb T}}

\newcommand{\bPU}{{\mathbb{PU}}}

\newcommand{\bZ} {{\mathbb Z}}

\newcommand{\bN} {{\mathbb N}}
\newcommand{\bQ} {{\mathbb Q}}
\newcommand{\bP} {{\mathbb P}}

\newcommand{\bK} {{\mathbb K}}

\newcommand{\ud}{{{\mathbb U}(d)}}
\newcommand{\sud}{{{\mathbb {SU}}(d)}}

\newcommand{\eps}{\epsilon}
\newcommand{\vareps}{\xi}

\newcommand{\mB}{\mathcal B}

\newcommand{\mE}{\mathcal E}

\newcommand{\mG}{\mathcal G}

\newcommand{\mO}{\mathcal O}
\newcommand{\mP}{\mathcal P}

\newcommand{\efA}{\EuFrak{A}}
\newcommand{\efB}{\EuFrak{B}}
\newcommand{\efC}{\EuFrak{C}}

\newcommand{\efF}{\EuFrak{F}}
\newcommand{\efG}{\EuFrak{G}}
\newcommand{\efH}{\EuFrak{H}}

\newcommand{\efP}{\EuFrak{P}}

\newcommand{\efR}{\EuFrak{R}}
\newcommand{\efS}{\EuFrak{S}}
\newcommand{\efT}{\EuFrak{T}}

\newcommand{\efc}{\EuFrak{c}}

\newcommand{\bo}{{\partial_0 b}}
\newcommand{\bl}{{\partial_1 b}}

\newcommand{\ad}{{\mathrm{ad}}}

\newcommand{\rs}{{\rho,\sigma}}
\newcommand{\sr}{{\sigma,\rho}}

\newcommand{\Calg}{{\bf C^*alg}}
\newcommand{\Ccat}{{\bf C^*cat}}
\newcommand{\obj}{{\bf obj \ }}

\newtheorem{thm}{Theorem}[section]
\newtheorem{cor}[thm]{Corollary}
\newtheorem{lem}[thm]{Lemma}
\newtheorem{prop}[thm]{Proposition}

\newtheorem{defn}[thm]{Definition}

\newtheorem{rem}[thm]{Remark}

\theoremstyle{definition}
\newtheorem{ex}{Example}[section]

\theoremstyle{remark}

\numberwithin{equation}{section}

\begin{document}

\author{{\sf Ezio Vasselli}
                         \\{\it \small{Dipartimento di Matematica University of Rome La Sapienza}}
                         \\{\sf ezio.vasselli@gmail.com}}

\title{Presheaves of symmetric tensor categories\\and\\nets of ${\mathrm{C}}^*$-algebras}
\maketitle

\begin{abstract}
Motivated by algebraic quantum field theory, we study presheaves of symmetric tensor categories
defined over the base of a space, intended as a spacetime.
Any section of a presheaf (that is, any "superselection sector", in the applications that we have in mind) 
defines a holonomy representation whose triviality is measured by Cheeger-Chern-Simons characteristic classes,
and a non-abelian unitary cocycle defining a Lie group gerbe.
We show that, given an embedding in a presheaf of full subcategories of the one of Hilbert spaces, 
the section category of a presheaf is a Tannaka-type dual of a locally constant group bundle (the "gauge group"), 
which may not exist and in general is not unique.
This leads to the notion of gerbe of \sC algebras, defined on the given base.
\end{abstract}

\tableofcontents
\markboth{Contents}{Contents}

\newpage
\section{Introduction.}
\label{intro}

Let $M$ be a spacetime and $\Delta$ denote a base generating the topology of $M$. 
The object of interest in algebraic quantum field theory is a \sC precosheaf $\efA$ defined over $\Delta$,
that is, a family of \sC algebras $\{ A_o \}_{o \in \Delta}$ with *-monomorphisms
\begin{equation}
\label{eq.00}
\jmath_{o'o} : A_o \to A_{o'}
\ \ : \ \
\jmath_{o''o'} \circ \jmath_{o'o} = \jmath_{o''o}
\ \ , \ \
\forall o \subseteq o' \subseteq o'' \in \Delta
\ ,
\end{equation}
where each \emph{fibre} $A_o$ is interpreted as the algebra of quantum observables localized within $o$.
This is what is usually called the \emph{observable net} (\cite[Chap.III]{Haa}).

\smallskip

In the case of the Minkowski spacetime $\Delta$ is a directed \emph{poset} (partially ordered set) 
when ordered under inclusion, so we can define the inductive limit
$\vec{A} := \lim_{o \subseteq o'} ( A_o , \jmath_{o'o} )$.
This is an important property at the mathematical level, because it implies that the set of \emph{sectors}
(the physically relevant Hilbert space representations of $\efA$) can be realized as a semigroup $T$
of *-endomorphisms of $\vec{A}$ defining a symmetric tensor category with simple unit. 
This allows one to reconstruct the \sC algebra of quantum fields, as a \sC crossed product 
\begin{equation}
\label{eq.01}
F := \vec{A} \rtimes T \ ,
\end{equation}
and the gauge group $G$, realized as the group of automorphisms of $F$ leaving $\vec{A}$ pointwise fixed.
Then a net $\efF$ of \sC subalgebras of $F$ is constructed and interpreted as the net of
(not necessarily observable) quantum fields. The pair $(\efF,G)$ is unique up to a suitable equivalence
and $T$ is characterized as the category of unitary representations of $G$ (\cite{DR89A,DR90}).

A class of "universal models" of (\ref{eq.01}) is given by the representation of the Cuntz algebra 
$\mO_d$, $d \in \bN$ (\cite{Cun77}), as the crossed product
\begin{equation}
\label{eq.01b}
\mO_d = \mO_G \rtimes \wa \rho \ .
\end{equation}
In the previous expression $\mO_G \subset \mO_d$ is the fixed-point \sC algebra under the action of the compact group 
$G$, and $\wa \rho$ is the semigroup generated by the canonical endomorphism 
$\rho \in {\bf end}\mO_G$ carrying the structure of a symmetric tensor category, see \cite{DR87} and \cite[Theorem 4.17]{DR89}. 
The relation with (\ref{eq.01}) is given by the fact that for any $\tau \in T$ 
there are mutually orthogonal partial isometries $\psi_1 , \ldots , \psi_d \in F$,
that induce $\tau$ on $\vec{A}$ in the sense that
\[
\tau(t) \ = \ \sum_k \psi_k t \psi_k^* \ \ , \ \ \forall t \in \vec{A} \ .
\]
The \sC algebra $F_\tau \subseteq F$ generated by $\psi_1 , \ldots , \psi_d$
contains the \sC algebra $\mO_\tau \subseteq \vec{A}$ generated by the elements of $\vec{A}$ that intertwine
powers of $\tau$, and there is an isomorphism $i : \mO_d \to F_\tau$ such that 
$i(\mO_G) = \mO_\tau$ and $i \circ \rho (t) = \tau \circ i (t)$, $\forall t \in \mO_G$.

\

Now, in conformal theory and general relativity we have spacetimes whose bases are not directed under inclusion.
This fact has consequences both at the level of the \sC precosheaf $\efA$, on which the operation of inductive limit now is not defined,
and on the structure of the set of sectors. 

A method to overcome the problem of the inductive limit has been given by Fredenhagen, 
who defined a \sC algebra $\wt A$ fulfilling the universal property of lifting Hilbert space representations of $\efA$ (see \cite{Fre}). 
This construction has two limits: the first is that $\wt A$ does not contain the informations necessary to describe the sectors affected 
by the topology of $M$ introduced by Brunetti and Ruzzi (\cite{BR08}), and the second is that in general it is not ensured that a 
\sC precosheaf can be faithfully represented on a Hilbert space (\cite[Example 5.8, Example A.9]{RV11}).

Passing to the structure of sectors, one can define semigroups 
$T_a \subseteq {\bf end}A_a$, $a \in \Delta$,
with elements *-endomorphisms \emph{localized} in some $e \subset a$ 
(see Example \ref{ex.Rob}, \cite[\S 27]{Rob}, \cite[\S 3.3]{GLRV01}).
These semigroups form a precosheaf $\efT$ encoding the tensor structure but that, unfortunately, 
does not contain all the informations necessary to characterize the set of sectors (see \cite[\S 3.3]{GLRV01}, \cite[\S 5.1]{BR08}).
%
%
A different approach is proposed in the preprint \cite{VasQFT}, 
where we show that the Brunetti-Ruzzi superselection structure $S$ is the category of sections of 
a \emph{presheaf} $\efS$ of symmetric tensor categories with simple units. 
The idea is that any fibre $S_a$ of $\efS$, $a \in \Delta$, can be interpreted as a superselection structure defined on $A_a$,
which, unlike $S$, is not affected by topological effects.
Our program is to construct the field algebras as suitable crossed products 
$F_a = A_a \rtimes S_a$, $a \in \Delta$,
and then to construct structure morphisms of the type (\ref{eq.00}) for the field algebras.
This would yield the reconstruction of quantum fields, that is currently an open problem (\cite[\S 8]{BR08}).

\

In the present paper we make the first steps in this direction at the mathematical level. 
The first is the analysis of the presheaf $\efS$ from the point of view of the Tannaka duality, 
that allows us to understand which are the natural dual objects defined by $\efS$
{\footnote{
Passing to the dual poset $\Delta'$ we can regard presheaves over $\Delta$ as precosheaves over $\Delta'$ and vice versa, 
so our results apply to precosheaves of symmetric tensor categories. 
%
%
}}.
The second step is the study of the \sC net bundles
{\footnote{A \emph{\sC net bundle} is a \sC precosheaf such that the structure morphisms (\ref{eq.00}) are isomorphisms, see \cite{RRV07}.}}
defined by sections of $\efS$, that play a role analogous to the above-mentioned \sC algebra $\mO_\tau$:
in particular, we study a version of (\ref{eq.01b}) in the setting of \sC net bundles.

\smallskip

For what concerns the Tannaka duality, as a dual object we obtain (quite naturally)
a precosheaf of compact groups $\efG$ acting on precosheaves of finite-dimensional Hilbert spaces, see Theorem \ref{thm.dual}.
Along this way, we prove that any section $\varrho$ of $\efS$ (that is, any Brunetti-Ruzzi sector in accord with \cite{VasQFT}) 
defines a compact Lie group 
$G_\varrho \subseteq \ud$, $d \in \bN$, 
and a holonomy representation 
\begin{equation}
\label{eq.02}
\chi_\varrho : \pi_1(M) \to NG_\varrho / G_\varrho \ ,
\end{equation}
where $NG_\varrho$ is the normalizer of $G_\varrho$ in $\ud$.
The morphism (\ref{eq.02}) is a complete invariant of the presheaf defined by tensor powers of $\varrho$ (Theorem \ref{thm.hol}). 
We interpret $\chi_\varrho$ as a flat principal $NG_\varrho / G_\varrho$-bundle
and assign to it Cheeger-Chern-Simons classes living in the odd cohomology of $M$ (see (\ref{eq.ccs})).
We point out that our duality is a \emph{Tannaka} duality, based on the use of an embedding
$I : \efS \to \efC$,
where $\efC$ is a presheaf of full subcategories of the one of Hilbert spaces.
Existence and uniqueness of $I$ and, consequently, of $\efG$ can be expressed in terms of lifts 
\[
\wt \chi_\varrho : \pi_1(M) \to NG_\varrho \ \ , \ \ \chi_\varrho = \wt \chi_\varrho \, {\mathrm{mod}} G_\varrho \ ,
\]
and therefore are not ensured, see Theorem \ref{cor.lift}.
As we shall see in a future paper, existence and uniqueness of the dual object can be restored
by considering \emph{gerbes over posets} in the sense of \S \ref{sec.concl},
that is, families of group isomorphisms fulfilling the precosheaf relations up to inner automorphisms.
This situation is analogous to the one of \cite{Vas09}, where similar results are proven for 
bundles of symmetric tensor \sC categories, nevertheless motivations, techniques and results diverge,
because now our scenario is the "geometry of posets" of \cite{RRV07}.

\smallskip

About the analogue of (\ref{eq.01b}), we observe in Remark \ref{rem_CCS00} that any section $\varrho$ of $\efS$ 
defines a \sC net bundle $\efA_\varrho$ whose fibres are constructed in the same way as $\mO_\tau$.
This leads us to study \sC net bundles $\efA$ with fibre $\mO_G$, and we classify them in terms of holonomy representations
$\chi : \pi_1(M) \to NG / G$,
see Theorem \ref{thm_OG}.
Any $\efA$ is endowed with a *-endomorphism $\varrho_* \in {\bf end}\efA$ defining 
a presheaf $\wa \varrho_*$ with fibres isomorphic to $\wa \rho$, 
and we show that solutions of the problem of finding a \sC net bundle $\efF$ with fibre $\mO_d$,
that plays the role of a crossed product
\begin{equation}
\label{eq.03}
\efF = \efA \rtimes \wa \varrho_* \ ,
\end{equation}
are in one-to-one correspondence with lifts of the holonomy $\chi$ (Theorem \ref{thm_FL08}); 
as desired, any solution $\efF$ determines a group precosheaf $\efG$ such that $\efF^\efG = \efA$.
Examples in which the pair $(\efF,\efG)$ does not exist or is not unique are given in \S \ref{sec_end}.
We shall show in a future work that (\ref{eq.03}) has always a unique solution
searching $\efF$ in the more general category of \emph{gerbes} of \sC algebras in the sense of (\ref{eq.CON3}),
and this will determine a gerbe of compact groups rather than a group precosheaf (\cite{VasX}).

\

The following sections are organized as follows.

\smallskip

In Section \ref{sec_nets} we recall the notions of precosheaf (net bundle) and presheaf (presheaf bundle)
over a poset $K$ with fibres in a category ${\bf C}$ 
(see \cite{RRV07}: as for \sC precosheaves, the term \emph{bundle} indicates that the structure morphisms are isomorphisms).
We illustrate the equivalence between the category of net bundles with fibres in ${\bf C}$ 
and the one of representations on ${\bf C}$ of the homotopy group $\pi_1(K)$,
realized by assigning to the given net bundle its holonomy representation (Theorem \ref{thm_NET00}). 
Then we introduce the notion of \emph{gauge action} of a group net bundle $\efG$ on a precosheaf, see (\ref{eq_GA00}).
Finally, we consider presheaves of categories and prove that for any presheaf $\efS$ there is a canonical presheaf bundle
with category of sections isomorphic to the one of $\efS$ (Prop.\ref{prop.NET1}).

In Section \ref{sec_end} we give the above-mentioned classification of \sC net bundles $\efA$ with fibre $\mO_G$
and characterize those that admit solutions $\efF$ of (\ref{eq.03}). 
These results are also stated in terms of duality between presheaves of symmetric tensor categories and group net bundles.

In Section \ref{sec_CCS} we prove a Tannaka duality for the category of sections of $\efS$ (Theorem \ref{thm.dual}). 
Other results are those concerning the holonomy representation (\ref{eq.02}), see Theorem \ref{thm.hol}, and 
the non-abelian cocycle defined by a section of $\efS$ (Theorem \ref{thm.gerbe}).

\smallskip

The appendix contains an introduction to flat and locally constant bundles with fibres in the category ${\bf C}$, 
that form a non-full subcategory, denoted by ${\bf lc}(M,{\bf C})$, of the one of locally trivial bundles on $M$ with fibres in ${\bf C}$. 
Making use of the holonomy representation, we illustrate how ${\bf lc}(M,{\bf C})$ is equivalent to 
the category of net bundles over $\Delta$ with fibres in ${\bf C}$.

\section{Notation and background.}
\label{sec_back}

In the present paper we shall use notions arising from contexts as algebraic topology, operator algebras and category theory, so to facilitate the reader we give some definitions.

\begin{itemize}
\item ${\bf aut}G$ is the automorphism group of the group $G$ and $\ad g \in {\bf aut}G$ is the inner automorphism induced by $g \in G$;
\item If $K$ is a set of indices and $B := \{ B_o \}_{o \in K}$, $B' := \{ B'_o \}_{o \in K}$ are families of sets, then we define the \emph{fibred product}
$B \times_K B' := \{ B_o \times B'_o \}_{o \in K}$;
\end{itemize}

\noindent {\bf Spaces.} With the term \emph{space} we mean a topological space. In particular, we denote an arcwise connected, locally compact, Hausdorff space by $M$ (that will play the role of a spacetime in the applications). We say that a base $\Delta$ generating the topology of $M$ is \emph{good} whenever any element of $K$ is arcwise and simply connected.

\

\noindent {\bf Operator algebras and (pointed) dynamical systems (\cite{Dix}).} 
A \sC algebra is a complex Banach *-algebra $A$ such that the \sC identity 
$\| v^*v \| = \| v \|^2$
is fulfilled for all $v \in A$. In the present paper we shall work with unital \sC algebras, 
that is, there is $1 \in A$ such that $1v=v1=v$, $\forall v \in A$, and we have the group of unitaries
$UA := \{ u \in A : uu^*=u^*u=1 \}$.
Any \sC algebra can be represented as a norm-closed *-algebra of operators on a complex Hilbert space.

A $*$-morphism $\phi : A \to A'$ is a linear map such that
$\phi(v^*w)=\phi(v)^* \phi(w)$, $\forall v,w \in A$.
%
The terms $*$-endomorphism, $*$-automorphism shall be used with the obvious meaning.
%
%
We shall consider only \emph{unital} *-morphisms, i.e. such that $\phi(1) = 1$.

A pair $(A,\rho)$, where $A$ is a \sC algebra and $\rho \in {\bf end}A$ a *-endomorphism,
is called \emph{\sC dynamical system}. 
A morphism $\phi : (A,\rho) \to (A',\rho')$ is a $*$-morphism $\phi : A \to A'$ 
fulfilling the relation $\rho' \circ \phi = \phi \circ \rho$. 
A \emph{pointed} \sC dynamical system is a triple 
$(A,\rho,\eps) \equiv A_{\rho,\eps}$, $\rho \in {\bf end}A$, $\eps \in A$; 
a morphism of pointed \sC dynamical systems $\phi : A_{\rho,\eps} \to A'_{\rho',\eps'}$ 
is a morphism $\phi : (A,\rho) \to (A',\rho')$ such that $\phi(\eps) = \eps'$. 
The automorphism groups of $(A,\rho)$ and $(A,\rho,\eps)$ are denoted by
${\bf aut}_\rho A$, ${\bf aut}_{\rho,\eps}A$, 
respectively.
We shall use the notion of pointed \sC dynamical system having in mind the case in which 
$\eps$ is a symmetry operator in the sense of \cite[\S 4]{DR89A}.


\

\noindent {\bf Categories.} We refer to \cite{ML} for notions of category theory,
and to \cite[\S 1]{DR89} for the notions of \sC category, direct sums and subobjects.
Given a category ${\bf C}$, we denote the set of objects by $\obj {\bf C}$ and the set of arrows by ${\bf arr} \ {\bf C}$. 
For any $\tau \in \obj {\bf C}$, we denote the group of invertible arrows in $(\tau,\tau)$ by ${\bf aut}\tau$ and the identity arrow by $1_\tau \in {\bf aut}\tau$.

For the notion of \emph{functor} 
$F : {\bf C} \to {\bf D}$
we refer to \cite[\S I.3]{ML}, from which the following definitions are taken.
The functor $F$ is said to be an \emph{isomorphism} whenever it is bijective both on the set of objects and on the sets of arrows, so it has an inverse functor,
and, in particular, we say that $F$ is an \emph{automorphism} whenever ${\bf C} = {\bf D}$.
We say that $F$ is an \emph{embedding} whenever it is injective on the sets of arrows, and that $F$ is \emph{full} whenever it is surjective on the sets of arrows.
Finally, $F$ is said to be an \emph{equivalence} whenever there is $F' : {\bf D} \to {\bf C}$ such that 
$F \circ F' \simeq id_{\bf D}$ and $F' \circ F \simeq id_{\bf C}$,
where $id_{\bf C}$ is the identity functor and $\simeq$ denotes a natural isomorphism (see \cite[\S I.4]{ML}).

Many categories that are considered in the present paper are \emph{topological}, that is: 
(1) Any $X \in \obj {\bf C}$ is a space; 
(2) Any set of arrows $(X,X')$ has elements continuous maps from $X$ to $X'$ and is endowed with a suitable topology; 
(3) The composition defines continuous maps on the sets of arrows.

The cases of interest are the following: 
(1) The category ${\bf Ban}$ with objects Banach spaces and arrows bounded linear operators, endowed with the norm topology on the objects and with the pointwise convergence topology on the sets of arrows. If $X \in \obj {\bf Ban}$, then ${\bf aut}X$ is the group of invertible linear operators;
(2) The category $\Calg$ with objects \sC algebras and arrows $*$-morphisms, with the topologies of ${\bf Ban}$. If $A \in \obj \Calg$, then ${\bf aut}A$ is the *-automorphism group;
(3) The category ${\bf Hilb}$ with objects Hilbert spaces and arrows bounded linear operators, with the topologies of ${\bf Ban}$. If $H \in \obj {\bf Hilb}$, then by definition ${\bf aut}H := UH$, the unitary group; 
(4) The category ${\bf TopGr}$ with objects topological groups and arrows continuous maps, endowed with the pointwise convergence topology on the sets of arrows. If $G \in \obj {\bf TopGr}$, then ${\bf homeo}G$ is the homeomorphism group.

\

\noindent {\bf Symmetric tensor categories.} Categories appear in of algebraic quantum field theory for questions related to the gauge group, under the form of superselection structures (\cite{DR90}). In the present paper we use \emph{tensor \sC categories} $T$, that means that the spaces of arrows $(\rs)$, $\rs \in \obj T$, are Banach spaces with an involutive structure $* : (\rs) \to (\sr)$ satisfying the \sC identity $\| t^* \circ t \| = \| t \|^2$, $\forall t \in (\rs)$, and endowed with a *-bifunctor 
$\otimes : T \times T \to T$, 
the tensor product, having identity object $\iota \in \obj T$;
to be concise, we write
$\rho \sigma := \rho \otimes \sigma$ for any $\rho , \sigma \in \obj T$.
We say that $T$ \emph{has a simple unit} whenever $(\iota,\iota) \simeq \bC$.
A further structure is the one of \emph{symmetry}, that is defined by a family of unitary arrows 
$\eps = \{ \eps_\rs  \in (\rho \sigma , \sigma \rho) \}$,
%
%
fulfilling
\begin{equation}
\label{eq.eps1}
(t' \otimes t) \circ \eps_{\rho,\rho'} = \eps_{\sigma,\sigma'} \circ (t \otimes t')
\ \ , \ \
\forall t \in (\rs) \ , \ t' \in (\rho',\sigma') \ ,
\end{equation}
and 
\begin{equation}
\label{eq.eps2}
\eps_{\rho,\sigma} \circ \eps_{\sigma,\rho} = 1_\rho
\ \ , \ \
\eps_{\iota,\rho} = \eps_{\rho,\iota} = 1_\rho
\ \ , \ \
\eps_{\rho \sigma,\tau} = (\eps_{\rho,\tau} \otimes 1_\sigma) \circ (1_\rho \otimes \eps_{\sigma,\tau})
\ ,
\end{equation}
for all $\rho,\sigma,\tau \in \obj T$.
We denote a symmetric tensor \sC category by $T_{\otimes,\eps}$, and we write 
${\bf aut}T_{\otimes,\eps}$ 
to indicate the group of automorphisms preserving tensor product and symmetry.

The category ${\bf Hilb}$, duals of compact groups (\cite[Example 1.2]{DR89}) and symmetric $*$-endomorphisms 
(\cite[\S 5]{DR89A}) are the models that we have in mind for symmetric tensor \sC categories. 
For convenience we recall the definition of these last ones: given the $C^*$-algebra $A$,
${\bf end}A$ becomes a \sC category with sets of arrows given by the Banach spaces
\begin{equation}
\label{eq.intro0}
(\rho,\sigma) := \{ t \in A \ : \ t \rho(v) = \sigma(v)t \ , \ \forall v \in A \}
\ \ , \ \ 
\forall \rho,\sigma \in {\bf end}A
\ ,
\end{equation}
and composition of arrows defined by the product of $A$; the tensor product is given by
\begin{equation}
\label{eq.intro1}
\rho \otimes \sigma := \rho \sigma := \rho \circ \sigma
\ \ , \ \
t \otimes t' := t \rho(t')
\ \ , \ \
\forall \rho,\sigma,\rho',\sigma' \in {\bf end}A
\ , \ 
t \in (\rho,\sigma) \ , \  t' \in (\rho',\sigma')
\ .
\end{equation}
If a tensor full subcategory of ${\bf end}A$ has symmetry $\eps$, 
then applying (\ref{eq.eps1},\ref{eq.eps2}) to this context we obtain the notion of \emph{permutation symmetry}
for a semigroup of *-endomorphisms of $A$ (see \cite[\S 5]{DR89A}).

\

\noindent {\bf Posets} are sets endowed with an order relation $\leq$. 
Sometimes it is convenient to regard a poset $K$ as the category
with objects elements of $K$ and sets of arrows defined by the order relation:
$(o,o')$, for all $o,o' \in K$, has one arrow when $o \leq o'$, otherwise it is the empty set.
In the present paper, the motivation of the notion of poset is given by good bases of manifolds endowed with the inclusion relation.

As in \cite{RR06,RRV07} we consider the simplicial set $\Sigma_*$ associated with $K$, which, at lower degrees, is defined as follows.
As a first step, we define the set of "points" $\Sigma_0(K) := K$.
The set $\Sigma_1(K)$ of $1$--simplices is given by triples 
\[
b := (\bo,\bl;|b|)
\ \ : \ \ 
\bo , \bl , |b| \in \Sigma_0(K) \  ,\ \bo,\bl \leq |b| \ .
\]
The $0$--simplices $\bo,\bl$ are called \emph{the faces} of $b$, whilst $|b|$ is called \emph{the support}, 
with the idea that $b$ is a "line" from $\bl$ to $\bo$.
%
%
The set $\Sigma_2(K)$ is given by quadruples $c := (\partial_0c,\partial_1c,\partial_2c,|c|)$ where each $\partial_kc$, $k=0,1,2$, is a $1$-simplex with $|\partial_k c| \leq |c|$ and such that 
\begin{equation}
\label{eq.dd}
\partial_{hk}c := \partial_h \partial_k c = \partial_k \partial_{h+1}c
\ \ , \ \
\forall h \geq k
\ .
\end{equation}
A path in $K$ is a finite sequence 
\[
\gamma = (b_n , \ldots , b_1)
\ \ : \ \ 
b_i \in \Sigma_1(K) \ , \ \partial_0b_i = \partial_1b_{i+1} \ , \ \forall i = 1 \ldots , n \ ;
\]
we define the ending and the starting points of $\gamma$ by
$\partial_0 \gamma := \bo_n$, $\partial_1 \gamma = \bl_1$,
respectively,
and write $\gamma : \partial_1\gamma \to \partial_0\gamma$.
Given $a,a' \in K$, we denote the set of paths of the type $\gamma : a \to a'$ by $K(a,a')$ and, in particular, we write $K(a) := K(a,a)$. For brevity we assume that posets are pathwise connected, that is, that for any pair $a,a'$ there is a path $\gamma : a \to a'$. Paths 
$\gamma = (b_n , \ldots , b_1)$, $\gamma' = (b'_n , \ldots , b'_1)$
with $\partial_1 \gamma = \partial_0 \gamma'$ can be composed in the obvious way:
\[
\gamma * \gamma' := ( b_n , \ldots , b_1 , b'_n , \ldots , b'_1 ) \ .
\]
There is a notion of \emph{deformation} on elements of $K(a,a')$ (see \cite[\S 2.2]{Ruz05}),
which yields an equivalence relation $\sim$. Thus we obtain the group
$\pi_1(K) := K(a) / \sim$ 
endowed with the operation of path composition, called the \emph{homotopy group based on $a$}.

The opposite poset of $K$ is given by the set $K' := K$ endowed with the opposite order relation 
$o' \leq' o \Leftrightarrow o \leq o'$.
There is an isomorphism $\pi_1(K) \simeq \pi_1(K')$, realized by observing that, exchanging the role of supports and faces we pass from a path in $K$ to a path in $K'$ (\cite[\S 2]{RRV07}).

Let $M$ be a space and $\Delta$ denote a good base of $M$. By \cite[Theorem 2.18]{Ruz05}, for each $x \in M$ and $a \in K$ with $x \in a$ there is an isomorphism
\begin{equation}
\label{eq_BK01}
\pi_1(\Delta) \to \pi_1(M)
\ \ , \ \
p \mapsto p_{top}
\ ,
\end{equation}
where $\pi_1(M)$ is the fundamental group with base point $x \in M$ and $\pi_1(\Delta)$ is the homotopy group of $\Delta$ based on $a$. 
By definition, the homotopy class $p_{top}$ has elements curves with support contained in $|p| := \cup_i |b_i|$, $p := (b_n , \ldots , b_1)$.

\

\noindent {\bf Groups.} Let $\Pi$, $G$ be groups. We denote the set of morphisms from $\Pi$ to $G$ by $\hom (\Pi,G)$. We say that $\chi , \chi' \in \hom ( \Pi , G )$ are \emph{equivalent} whenever there is $g \in G$ such that 
\[
\chi (p) \ = \  g \chi'(p) g^{-1} \ \ , \ \ \forall p \in \Pi \ .
\]
This defines an equivalence relation and we denote the set of equivalence classes by $\unl{\hom}(\Pi,G)$.
Given a category ${\bf C}$, we define the category ${\bf hom}(\Pi,{\bf C})$ in the following way: the set of objects is defined by 
$\{ \chi \in \hom (\Pi,{\bf aut}X) : X \in \obj {\bf C} \}$, 
and the set of arrows $(\chi,\chi')$, $\chi , \chi' \in \hom (\Pi,{\bf aut}X)$, is given by the intertwiners
\[
f \in (X,X') 
\ \ : \ \
\chi' (p) \circ f = f \circ \chi(p)
\ \ , \ \qquad
\forall p \in \Pi
\ .
\]

\section{Nets and presheaves over posets.}
\label{sec_nets}

Nets are a natural generalization of the notion of net of \sC algebras, 
in the sense that generic categories are considered instead of $\Calg$ and an 
abstract poset $K$ is considered instead of the base of double cones of the Minkowski space. 
The correct term should be the one of \emph{precosheaf}, 
nevertheless we prefer to maintain the terminology used in algebraic quantum field theory.
Unless otherwise stated we have topological categories ${\bf C}$, 
so any $X \in \obj {\bf C}$ is a space.

\

\noindent {\bf Nets.} Let us consider a collection 
$B := \{ B_o \}_{o \in K}$ 
of objects of ${\bf C}$ and a family $\jmath := \{ \jmath_{o'o} \}$ of arrows
$\jmath_{o'o} \in ( B_o , B_{o'} )$, $o \leq o' \in K$,
fulfilling the \emph{net relations}
\begin{equation}
\label{eq_jF1}
\jmath_{o'' o} = \jmath_{o'' o'} \circ \jmath_{o'o}
\ \ , \ \
\forall o \leq o' \leq o'' \in K
\ .
\end{equation}
We call \emph{net} the pair
$\efB := (B,\jmath)_K$;
the objects $B_o$, $o \in K$, are called the \emph{fibres of $\efB$} and $\jmath$ is called \emph{the net structure}.
Note that if we regard $K$ as a category, see \S \ref{sec_back}, then (\ref{eq_jF1}) says that $\efB$
is a functor from $K$ to ${\bf C}$.

\smallskip

For any $S \subset K$, the \emph{restriction} of $\efB$ to $S$ is given by the net
$\efB_S$
obtained by considering the families $\{ B_o , o \in S \}$, $\{ \jmath_{o'o} , o,o' \in S \}$.
In particular, given $a \in K$ we set 
$K^a := \{ o \in K : o \leq a \}$, $K_a := \{ o \in K : a \leq o \}$
and
$\efB_{\leq a} := \efB_{K^a}$, $\efB_{\geq a} := \efB_{K_a}$.

\begin{rem}{\it 
In the present paper we assume that the maps $\jmath_{o'o}$ are injective,
so we have $*$-monomorphisms when ${\bf C} = \Calg$, isometries when ${\bf C} = {\bf Hilb}$ and so on.
} \end{rem}

Let $\efB' := (B',j')_K$ be a net. A \emph{morphism} $\phi : \efB \to \efB'$ is given by a family of arrows $\phi := \{ \phi_o \in ( B_o , B'_o ) \}$ fulfilling
\begin{equation}
\label{eq_nets01}
j'_{o'o} \circ \phi_o = \phi_{o'} \circ \jmath_{o'o}
\ \ , \ \
\forall o \leq o'
\ .
\end{equation}
Note that, passing to the point of view that a net is a functor $\efB : K \to {\bf C}$, 
we have that $\phi$ defines a natural transformation in the sense of \cite[\S I.4]{ML}.

\smallskip

A morphism $\alpha : \efB \to \efB'$ is said to be an \emph{isomorphism} whenever there is $\alpha' : \efB \to \efB'$ such that $\alpha' \circ \alpha = id_\efB$, $\alpha \circ \alpha' = id_{\efB'}$. In this way, we have the category
${\bf net}(K,{\bf C})$,
with objects nets having fibres in $\obj {\bf C}$ and arrows the above defined morphisms. 
Given a net $\efB$ we write ${\bf end}\efB := (\efB,\efB)$.

Let $\efB = (B,\jmath)_K$ be a net and $A := \{ A_o \subseteq B_o \}$ a family stable under $j$, i.e.
$\jmath_{o'o}(A_o) \subseteq A_{o'}$, $\forall o \leq o'$.
Then defining $\imath := \{ \jmath_{o'o} |_{A_o} \}$ yields the net $\efA = (A,\imath)_K$ and the inclusion morphism $I \in (\efA,\efB)$. 
We say that $\efA$ is a \emph{subnet of $\efB$}.

\begin{ex}{\it 
\label{ex_nets01}
Let $X \in \obj {\bf C}$. The \textbf{constant net with fibre} $X$ is defined by $\efB_0 := (B_0,\jmath_0)_K$, where
$B_0 := \{ B_o \equiv X \}$ 
and
$\jmath_{0,o' o}(v) := v$, $\forall o \leq o' \in K$, $v \in B_o = X$. 
A net $\efB$ is said to be \textbf{trivial} whenever there is an isomorphism $\alpha \in (\efB,\efB_0)$.
} \end{ex}

\begin{ex}{\it 
Let $M$ denote the Minkowski spacetime and $\Delta_{dc}$ the base of double cones generating the topology of $M$. 
Any net in the sense of \cite[Chap.III]{Haa} is a net of \sC algebras over $\Delta_{dc}$ 
in the sense of the present paper.
} \end{ex}

\begin{ex}{\it 
\label{ex_Hilb}
Let $K$ be a poset. Objects of ${\bf net}(K,{\bf Hilb})$ are called nets of Hilbert spaces and form in the natural way a symmetric tensor \sC category with direct sums, subobjects and unit given by the constant net $\efH_0 := (H_0,\jmath_0)$ with fibre $\bC$. When $K$ is connected $\efH_0$ has a simple unit, that is, $(\efH_0,\efH_0) \simeq \bC$. 
Nets of infinite-dimensional Hilbert spaces appear in \cite[\S 2]{BR08}.
} \end{ex}

\

\noindent {\bf Sections.} A \emph{section} of $\efB$ is given by a family $s := \{ s_o \in B_o \}$ such that 
$\jmath_{o'o}(s_o) = s_{o'}$,
$\forall o \leq o'$.
We denote the set of sections of $\efB$ by $\wt \efB$. Note that $\wt \efB$ inherits structure from the category ${\bf C}$, i.e. $\wt \efB$ is a group, a Hilbert space, a \sC algebra whenever ${\bf C} = {\bf TopGr}$, ${\bf Hilb}$, $\Calg$. For example, when ${\bf C} = \Calg$, at varying of $o$ in $K$ we define
\[
\left\| s \right\| := \sup_o \left\| s_o \right\|
\ , \
(ss')_o := s_o s'_o
\ , \
(s + \lambda s')_o := s_o + \lambda s'_o
\ , \
s^*_o := (s_o)^*
\ ,
\]
where $s,s' \in \wt \efB$, $\lambda \in \bC$. Since $K$ is pathwise connected and since (in the cases of interest) any $\jmath_{o'o}$ is isometric, the above defined norm function is constant at varying of $o$ in $K$. 

\

\noindent {\bf Presheaves.} A \emph{presheaf} over $K$ is given by a pair $\efS = (S,r)^K$, 
where $S$ is a family of objects of a category ${\bf C}$ and $r = \{ r_{aa'} , a \leq a' \}$, is a family of 
injective arrows fulfilling the presheaf relations
\[
r_{aa'} \circ r_{a'a''} = r_{aa''}
\ \ , \ \
\forall a \leq a' \leq a'' \in K \ .
\]
Note that requiring the injectivity property for $r$ makes presheaves over posets quite different from
those used in algebraic topology, where, in general, the restriction maps are not injective. 
Rather, geometric counterparts of presheaves over posets are
\emph{locally constant presheaves} in the sense of \cite[\S II.10,II.13]{BT}, 
as it can be proven reasoning as in Appendix \ref{sec_lc}.

A section of $\efS$ is given by a family 
$\varrho = \{ \varrho_o \in S_o : r_{oo'}(\varrho_{o'}) = \varrho_o , \forall o \leq o' \}$.
A presheaf morphism $\eta : \efS \to \efS'$ is given by a family of arrows 
$\eta_a : S_a \to S'_a$, $a \in K$,
such that
$r_{aa'} \circ \eta_{a'} = \eta_a \circ r_{aa'}$, $\forall a \leq a'$.
This yields the category ${\bf psheaf}(K,{\bf C})$ of presheaves over $K$ with coefficients in ${\bf C}$.
A presheaf can be regarded in the obvious way as a net on the opposite poset $K'$, and this yields isomorphisms
\[
{\bf net}(K,{\bf C}) \ \simeq \ {\bf psheaf}(K',{\bf C})
\ \ , \ \ 
{\bf net}(K',{\bf C}) \ \simeq \ {\bf psheaf}(K,{\bf C})
\ .
\]
\noindent {\bf Net bundles, flat connections and cocycles.} We describe some results of \cite[\S 3-4]{RR06} from a categorical viewpoint. The following considerations also hold for nets with fibres that are objects of generic categories, as the category of small \sC categories with arrows *-functors (which shall be used in the sequel).

A net $\efB = (B,\jmath)_K$ is said to be a \emph{net bundle} whenever $\jmath_{o'o}$ is an isomorphism for all $o \leq o'$,
and in this case it is customary to fix $a \in K$ and define the \emph{standard fibre} $X := B_a$. In the sequel, a net bundle with standard fibre $X$ will be denoted by
\[
\efB := (B,j;X)_K \ .
\]
To be concise, when $\efB$ is a net bundle we write 
$\jmath_{oo'} := \jmath_{o'o}^{-1}$, $o \leq o'$. 
We denote the full subcategory of ${\bf net}(K,{\bf C})$ with objects net bundles by ${\bf bun}(K,{\bf C})$; in particular, we denote the set of net bundles with standard fibre $X$ by ${\bf bun}(K,X)$ and the associated set of isomorphism classes by $\unl{\bf bun}(K,X)$.
Analogous notations and terminology hold for presheaf bundles, which form the category 
${\bf pbun}(K,{\bf C})$.
Note that we have a further, canonical isomorphism
\begin{equation}
\label{eq.np}
{\bf bun}(K,{\bf C}) \to {\bf pbun}(K,{\bf C})
\ \ , \ \
(B,\jmath)_K \mapsto (B,r)^K \ : \ r_{oo'} := \jmath_{oo'} \ , \ o \leq o'
\ .
\end{equation}
%
%
%
%
The \emph{flat connection} of the net bundle $\efB = (B,\jmath)_K$ is given by the family $Z$ of isomorphisms
\begin{equation}
\label{eq_FC01}
Z(b) : B_\bl \to B_\bo
\ \ , \ \
Z(b) := \jmath_{\bo,|b|} \circ \jmath_{|b|,\bl}
\ \ , \ \
b \in \Sigma_1(K)
\ ,
\end{equation}
which, by (\ref{eq.dd}), satisfies the $1$--cocycle relations
\begin{equation}
\label{eq_FC02}
Z(\partial_0c) \circ Z(\partial_2c) = Z(\partial_1c)
\ \ , \ \
\forall c \in \Sigma_2(K)
\ .
\end{equation}
Fixing a family of isomorphisms $i_o \in (B_o,X)$, $o \in K$, yields the map
\begin{equation}
\label{eq_FC03}
z : \Sigma_1(K) \to {\bf aut}X
\ \ , \ \
z (b) := i_\bo \circ Z(b) \circ i_\bl^{-1}
\ \ , \ \
b \in \Sigma_1(K)
\ .
\end{equation}
Clearly $z$ satisfies (\ref{eq_FC02}), i.e. it is an ${\bf aut}X$-cocycle in the sense of \cite[Eq.27]{RR06}. 
We now extend $Z$ to generic paths: for any $\gamma := b_n * \cdots * b_1$, we define
\begin{equation}
\label{eq_FC04A}
Z(\gamma) := Z(b_n) \circ \cdots \circ Z(b_1) : B_{\bl_1} \to B_{\bo_n} \ .
\end{equation}
In particular, if $\partial_1 \gamma = \partial_0 \gamma = a$ then $Z$ yields the map
\begin{equation}
\label{eq_FC04}
\dot{\chi} : K(a) \to {\bf aut}X
\ \ , \ \
\dot{\chi}(\gamma) := Z(\gamma)
\ .
\end{equation}
Reasoning as in \cite[Lemma 2.6]{Ruz05} we can prove that $\dot{\chi}$ induces a map $\chi : \pi_1(K) \to {\bf aut}X$, that we call the \emph{holonomy representation of $\efB$}. The image of $\chi$ is called \emph{the holonomy group of $\efB$ on $a$} and is denoted by $H_a(\efB)$. By \cite[Lemma 4.27]{RR06}, a change $a \to a'$ of the standard fibre yields a holonomy group conjugated with $H_a(\efB)$ in ${\bf aut}X$.

Let $\efB' = (B',j';X')_K$ be a net bundle; we write the map 
\[
\dot{\chi}' \, : \, K(a) \to {\bf aut}X'
\ \ , \ \
\dot{\chi}'(p) \, := \,
j'_{\bo_n,|b_n|} \circ j'_{|b_n|,\bl_n}
\circ \cdots \circ
j'_{\bo_1,|b_1|} \circ j'_{|b_1|,\bl_1}
\ ,
\]
defining the holonomy representation $\chi'$. If $t \in (\efB,\efB')$, then using (\ref{eq_nets01}) we find
\begin{equation}
\label{eq_NETGR}
t_a \circ \dot{\chi}(\gamma) = \dot{\chi}'(\gamma) \circ t_a
\ \ , \ \qquad
\forall \gamma : a \to a
\ .
\end{equation}
Thus $t_a$ is an intertwiner from $\chi$ to $\chi'$.

\

\noindent {\bf The equivalence with ${\bf hom}(\pi_1(K),{\bf C})$.} 
We are ready to illustrate the equivalence between the category of net bundles and that of representations of $\pi_1(K)$. 
For a detailed proof see \cite[Prop.3.11]{RR06} and \cite[Theorem 22(i)]{RRV07}, anyway for convenience we describe explicitly the involved maps.
\begin{thm}
\label{thm_NET00}
For any poset $K$ and for any category ${\bf C}$ there is an equivalence
\begin{equation}
\label{eq_NET01}
{\bf hom} ( \pi_1(K) , {\bf C} ) \to {\bf bun}(K,{\bf C}) 
\ \ , \ \
\chi \mapsto \efB_\chi
\ .
\end{equation}
\end{thm}

\begin{proof}[Sketch of the proof.]
In the previous lines we defined a functor from ${\bf bun}(K,{\bf C})$ to ${\bf hom}( \pi_1(K) , {\bf C} )$, assigning to any net bundle $\efB$ its holonomy representation $\chi$ and to any morphism $t \in (\efB,\efB')$ the intertwiner $t_a \in (\chi,\chi')$.
To prove the theorem we give, up to isomorphism, an inverse of this functor. To this end, using the fact that $K$ is connected, given $a \in K$ we fix paths
\[
\gamma_{ao} : o \to a \ \ , \ \ \forall o \in K \ .
\]
This operation is called a \emph{fixing of the path frame of $K$}, see \cite[\S 6]{BR08}. Moreover, we note that for any $o \leq o'$ the $1$-simplex $b_{o'o} := ( o' , o ; o' )$ can be regarded as a path in $K(o,o')$. By elementary properties of deformations, we have the equivalence
\begin{equation}
\label{eq_nets02}
b_{o'' o'} * b_{ o' o } \ \sim \ b_{ o'' o }
\ \ , \ \
o \leq o' \leq o''
\end{equation}
(see \cite[\S 2.2]{Ruz05}).
We now consider a group morphism $\chi : \pi_1(K) \to {\bf aut}X$, $X \in \obj {\bf C}$, and construct an associated net bundle. 
For any $o \in K$ we define $B_o := X$ and set $B := \{ B_o \}_{o \in K}$. To construct a net structure on $B$ we define
\begin{equation}
\label{eq_nets03}
\jmath_{o'o} \ := \ \chi ( \gamma_{ao'} * b_{o'o} * \gamma_{ao}^{-1} )
\ \ , \ \
o \leq o'
\ .
\end{equation}
By homotopy invariance of $\chi$, and using (\ref{eq_nets02}), we have 
$\jmath_{o'' o'} \circ \jmath_{o'o} = \jmath_{o'' o}$, $o \leq o' \leq o''$,
and this yields the desired net bundle $\efB_\chi = (B,j;X)_K$. 
\end{proof}

Let $G \subset {\bf aut}X$ be a subgroup. We say that \emph{the structure group of $\efB \in {\bf bun}(K,X)$ has a reduction to $G$} whenever there is a holonomy representation associated to $\efB$ with image $G$. Clearly, $\efB$ is trivial if, and only if, there is a reduction to the trivial group.

\

\noindent {\bf Principal net bundles.} Elements of ${\bf bun}(K,{\bf TopGr})$ are called group net bundles. 
An interesting subcategory of ${\bf bun}(K,{\bf TopGr})$ is the one of \emph{principal net bundles}, that we define in the following way.
Let $G$ be a topological group and ${\bf homeo}G$ denote the homeomorphism group of $G$. Then $G$ acts on itself by right multiplication and this yields an injective map $R : G \to {\bf homeo}G$. We say that $\efP \in {\bf bun}(K,G)$ is a \emph{$G$-principal net bundle} whenever it admits a holonomy representation $\chi : \pi_1(K) \to R(G) \simeq G$. To be concise, in the sequel we will identify $G$ with $R(G)$.
Let us denote the set of $G$-principal net bundles by ${\bf Pr}(K,G)$. By definition, there is a one-to-one correspondence
\begin{equation}
\label{eq_PR}
\hom ( \pi_1(K),G ) \to {\bf Pr}(K,G)
\ \ , \ \
\chi \mapsto \efP_\chi
\ .
\end{equation}
We say that $\efP,\efP' \in {\bf Pr}(K,G)$ are equivalent \emph{as $G$-principal net bundles} whenever there is $g \in G$ such that $\chi'(p) = g \chi(p) g^{-1}$, $\forall p \in \pi_1(K)$. We denote the set of equivalence classes of $G$-principal net bundles by $\unl{\bf Pr}(K,G)$; by definition there is a one-to-one correspondence 
\[
\unl{\bf Pr}(K,G) \ \simeq \ \unl \hom (\pi_1(K),G) \ .
\]
By Theorem \ref{thm_NET00}, any net bundle $\efB \in {\bf bun}(K,X)$ has an associated principal net bundle $\efP \in {\bf Pr}(K,{\bf aut}X)$, defined by the holonomy of $\efB$.
For details on principal net bundles see \cite[\S 4]{RRV07}.

\

\noindent {\bf Gauge actions.} Let $\efG := (Y,i;G)_K$ be a group net bundle. A \emph{gauge $\efG$-action} on the net $\efB$ is given by a family  
$\alpha_o : Y_o \to {\bf aut}B_o$, $o \in K$,
of continuous group morphisms fulfilling the relations
\begin{equation}
\label{eq_GA00}
\{ \alpha_{o'} (i_{o'o} (y)) \} \circ \jmath_{o'o}
=
\jmath_{o'o} \circ \alpha_o (y)
\ \ , \ \
\forall y \in Y_o \simeq G
\ , \
o \leq o'
\ .
\end{equation}
This yields the map
$\alpha : Y \times_K B \to B$,
$\alpha(y,v) := \{ \alpha_o(y) \} (v)$,
$\forall o \in K$, $y \in Y_o$, $v \in B_o$.
In the sequel we shall denote a gauge $\efG$-action by
\[
\alpha : \efG \times_K \efB \to \efB \ .
\]
We say that $\alpha$ is \emph{faithful} whenever $\alpha_o$ is injective for any $o \in K$.
Any section $g \in \wt \efG$ defines an automorphism 
\[
\alpha_g \in (\efB,\efB)
\ \ , \ \
\alpha_{g,o} (v) := \alpha ( g_o , v )
\ \ , \ \
v \in B_o
\ .
\]
When $\efG$ is trivial 
%
%
we have $G \simeq \wt \efG$ and we say that $\alpha$ is a \emph{global} gauge action. Clearly, a global gauge action defines group morphisms
\[
\alpha : G \to {\bf aut} \efB
\ \ , \ \
\wt \alpha : G \to {\bf aut} \wt \efB
\ .
\]
For example, if ${\bf C} = \Calg$ then $\wt \alpha$ is a strongly continuous action on the \sC algebra $\wt \efB$.

The pair $( \efB , \alpha )$ is called \emph{$\efG$-net}. Let $(\efB',\alpha')$ be a $\efG$-net and $t \in (\efB,\efB')$ a morphism; we say that $t$ is a \emph{$\efG$-morphism} whenever 
\begin{equation}
\label{eq_G01}
t_o \circ \alpha_o (y) = \alpha'_o(y) \circ t_o
\ \ , \ \
\forall o \in K \ , \ y \in Y_o \ ,
\end{equation}
and denote the set of $\efG$-morphisms by $(\efB,\efB')_\efG$.
In this way we have the subcategory ${\bf net}_\efG(K,{\bf C})$ of ${\bf net}(K,{\bf C})$ with objects $\efG$-bundles and arrows $\efG$-morphisms.
Given the $\efG$-net $(\efB,\alpha)$ and $o \in K$ we denote the space of fixed points w.r.t. $\alpha_o$ by $B_o^\alpha$; it is easily verified that $B^\alpha := \{ B_o^\alpha \}$ is stable under the net structure $j$, and we denote its restriction to $B^\alpha$ by $j^\alpha$. This yields the \emph{fixed-point subnet}
$\efB^\alpha := (B^\alpha,j^\alpha)_K$.

\begin{ex}
\label{ex_GHnB}
{\it 
Let $\efG = (Y,i;G)_K$ be a group net bundle. Objects of ${\bf bun}_\efG(K,{\bf Hilb})$ are called $\efG$-Hilbert net bundles.
%
%
%
} \end{ex}

\noindent {\bf Locally trivial gauge actions.} 
Let $\efB = (B,j;X)$ be a net bundle and $G := {\bf aut}X$. Then $\efB$ defines the group net bundle 
\begin{equation}
\label{eq_NET04}
{\bf aut} \efB := ( {\bf aut}B , \jmath_* ; G)_K
\ ,
\end{equation}
where 
${\bf aut}B := \{ {\bf aut}B_o \}$
and
$\jmath_{*,o' o} : {\bf aut}B_o \to {\bf aut}B_{o'}$, 
$\jmath_{*,o' o}(u) := \jmath_{o'o} \circ u \circ \jmath_{oo'}$,
$\forall o \leq o'$.
If $\chi \in \hom ( \pi_1(K), G )$ is the holonomy representation of $\efB$ then ${\bf aut}\efB$ has holonomy representation
\begin{equation}
\label{eq_FL07}
\ad \chi : \pi_1(K) \to {\bf aut}G
\ \ , \ \
\{ \ad \chi(p) \} (\beta) := \chi(p) \circ \beta \circ \chi(p)^{-1}
\ \ , \ \
p \in \pi_1(K)
\ , \
\beta \in G
\ .
\end{equation}
Any gauge $\efG$-action $\alpha$ on $\efB$ can be regarded, by (\ref{eq_G01}), as a morphism $\alpha \in ( \efG , {\bf aut}\efB )$.

\begin{ex}{\it 
Let $\efH = (H,j ; \bC^d )_K$ be a Hilbert net bundle. We denote the associated net bundle of unitary automorphisms by
${\bf U}\efH := (UH,\jmath_* ; \ud )_K$.
If $\chi \in \hom(\pi_1(K),\ud)$ is the holonomy representation of $\efH$, then ${\bf U}\efH$ has holonomy representation $\ad \chi \in \hom ( \pi_1(K),{\bf aut}\ud )$ defined by adjoint action.
} \end{ex}

\begin{lem}
\label{lem_LTGA1}
Let $\efB = (B,j;X)_K$ be a net bundle. Then faithful gauge actions on $\efB$ are in one-to-one correspondence with reductions of the structure group of $\efB$.
%
%
\end{lem}

\begin{proof}
Let $G$ be an automorphism group of $X$, $NG$ denote the normalizer of $G$ in ${\bf aut}X$ and $\chi \in \hom (\pi_1(K),NG)$ a reduction of the structure group of $\efB$. Defining $\chi'(p) := \ad \chi (p)|_G$, $p \in \pi_1(K)$, yields a morphism $\chi' \in \hom (\pi_1(K),{\bf aut}G)$. Clearly, the inclusion $i : G \to {\bf aut}X$ fulfils the relation $\ad \chi(p) \circ i = i \circ \chi'(p)$, $\forall p \in \pi_1(K)$, thus it defines an intertwiner $i \in (\chi',\ad \chi)$. Applying Theorem \ref{thm_NET00} we obtain a group net bundle $\efG \in {\bf bun}(K,G)$ and the desired monomorphism $i_* \in ( \efG , {\bf aut} \efB )$.
On the converse, let $\efG = (Y,i;G)_K$ denote a group net bundle and $\alpha \in ( \efG , {\bf aut}\efB )$ a monomorphism. Fixing $a \in K$ we define $G_a := \alpha_a(Y_a) \simeq G$ and ${\bf aut}B_a \equiv {\bf aut}X$. Now, we have $\jmath_{*,o' o} \circ \alpha_o (y) = \alpha_{o'} \circ i_{o'o}(y)$, $y \in Y_o$, $o \leq o'$, thus $\jmath_{o'o}(\alpha_o(Y_o)) = \alpha_{o'}(Y_{o'})$. This implies, evaluating over paths $p \in \pi_1(K)$,
\[
\{ \ad \chi(p) \} (g) 
\ = \ 
\chi(p) \cdot g \cdot \chi(p)^{-1} = g' \in G_a
\ \ , \ \
\forall g \in G_a
\ .
\]
In other terms, $\chi(p) \in NG_a$, and we conclude that $\chi$ takes values in $NG_a$.
\end{proof}

\

\noindent {\bf Nets and presheaves of \sC categories.} The concept of net can be categorified, see \cite[\S 27]{Rob}. A \emph{net of \sC categories} is given by a pair $\efC = (C,\jmath)_K$, where $C := \{ C_o , o \in K \}$ is a family of \sC categories and $j := \{ \jmath_{o'o}  , o \leq o' \}$ is a family of embeddings fulfilling the relations $\jmath_{o'' o'} \circ \jmath_{o'o} = \jmath_{o''o}$, $o \leq o' \leq o''$. We say that $\efC$ is \emph{full} whenever any $\jmath_{o'o}$, $o \leq o'$, is full, and that $\efC$ is a \emph{net bundle} whenever any $\jmath_{o'o}$ is an isomorphism.
Any net of \sC categories $\efC = (C,\jmath)_K$ defines the \emph{\sC category of sections} $\wt \efC$ with objects families
\begin{equation}
\label{eq_rho}
\varrho = \{ \varrho_o \in \obj C_o \}_{o \in K}
\ : \ 
\varrho_{o'} = \jmath_{o'o} \varrho_o
\ , \  
\forall o \leq o' 
\ ,
\end{equation}
sets of arrows
\[
(\varrho,\varsigma)
=
\{
t := \{ t_o  \in ( \varrho_o , \varsigma_o )  \}_{o \in K}
\ : \
t_{o'} = \jmath_{o'o}(t_o)
\ , \
\forall o \leq o'
\}
\ \ , \ \
\varrho,\varsigma \in \obj \wt \efC
\ ,
\]
and composition, \sC structure defined in the natural way (in particular, we define the norm $\| t \| := \sup_o \| t_o \|$, $t \in (\varrho,\varsigma)$).
For each $\rs \in \obj \wt \efC$, the net structure $j$ induces bounded linear maps
\begin{equation}
\label{eq_rhosigma}
\jmath_{o'o} : (\varrho_o,\varsigma_o) \to (\varrho_{o'},\varsigma_{o'})
\ \ , \ \
o \leq o'
\ ,
\end{equation}
defining a net of Banach spaces $\efC_{\varrho,\varsigma}$, 
thus elements of $(\varrho,\varsigma)$ can be regarded as sections of $\efC_{\varrho,\varsigma}$.
A \emph{morphism} from $\efC$ to the net $\efC' = (C',j')_K$ is given by a family of *-functors $\phi = \{ \phi_o : C_o \to C'_o \}$ such that 
$\phi_{o'} \circ \jmath_{o'o} = j'_{o'o} \circ \phi_o$, $o \leq o'$.
The morphism $\phi$ induces the *-functor
$\wt \phi : \wt \efC \to \wt \efC'$, 
$\wt \phi (\varrho) := \{ \phi_o(\varrho_o) \}$,
$\wt \phi(t) := \{ \phi_o(t_o) \}$,
$t \in (\varrho,\varsigma)$.

\

Analogous properties hold for the notion of presheaf of categories $\efC = (C,r)^K$
and we do not list them, limiting ourselves to observe that any
$r_{aa'} : C_{a'} \to C_a$, $a \leq a'$, 
is, by definition, an embedding. We denote the categories of nets, presheaves, net bundles, presheaf bundles 
of \sC categories over $K$, respectively, by
\[
{\bf net}^\uparrow(K,\Ccat)  \ \ , \ \
{\bf psheaf}^\uparrow(K,\Ccat) \ \ , \ \
{\bf bun}^\uparrow(K,\Ccat)    \ \ , \ \
{\bf pbun}^\uparrow(K,\Ccat)   \ .
\]
In particular, given the \sC category ${\bf C}$ we denote the set of net bundles
(resp. presheaf bundles) with fibres isomorphic to ${\bf C}$ by 
\[
{\bf bun}^\uparrow(K,{\bf C})
\ \ , \ \
{\bf pbun}^\uparrow(K,{\bf C})
\ .
\]
The above notation should not give rise to confusion with ${\bf bun}(K,{\bf C})$ which indicates,
instead, the category of net bundles whose fibres are objects of ${\bf C}$.
As for ordinary nets, we have isomorphisms
\[
{\bf psheaf}^\uparrow(K,\Ccat) \simeq {\bf net}^\uparrow(K',\Ccat) 
\ \ , \ \
{\bf bun}^\uparrow(K,\Ccat)  \simeq {\bf pbun}^\uparrow(K,\Ccat) 
%
%
\ .
\]
For future convenience the next results are expressed in terms of presheaves,
but they could be given as well in terms of nets.
\begin{prop}[The presheaf bundle trick]
\label{prop.NET1}
Let $\efC := (C,r)^K$ be a presheaf of \sC categories. 
Then there is a canonical morphism
\[
P : {}_\beta \efC \ \to \ \efC \ ,
\]
such that ${}_\beta \efC = ({}_\beta C , {}_\beta r)^K$ is a presheaf bundle
and $\wt P : \wt{_\beta \efC} \to \wt \efC$ is an isomorphism.
\end{prop}

\begin{proof}
To make $r$ bijective on the arrows we select a suitable subpresheaf of $\efC$.
To this end we consider the \sC category $\wt \efC_{\geq a}$ of sections of the restriction 
$\efC_{\geq a}$, $a \in K$,
and note that there is an evaluation functor
\[
\pi_a : \wt \efC_{\geq a} \to C_a
\ \ , \ \
\varrho \mapsto \varrho_a
\ , \
t \mapsto t_a
\ \ , \ \
\forall \varrho , \varsigma \in \obj \wt \efC_{\geq a}
\ , \
t \in (\varrho,\varsigma)
\ .
\]
We set
$_\phi C_a := \pi_a(\wt \efC_{\geq a})$, $\forall a \in K$.
Let $e \leq a$. By definition an arrow of $_\phi C_e$ is of the type
$t_e \in (\varrho_e,\varsigma_e)$,
where
$\varrho,\varsigma \in \obj \wt \efC_{\geq e}$ and $t \in (\varrho,\varsigma)$. 
But
$t_e = r_{ea}(t_a)$
and by definition $t_a \in {}_\phi (\varrho_a,\varsigma_a)$; this implies that 
$r_{ea}$ is a surjection on ${}_\phi (\varrho_e,\varsigma_e)$, $\forall e \leq a$ 
(note that $r_{ea}$ is also injective by definition),
and we define $_\phi r_{ea}$ as the restriction of $r_{ea}$ to $_\phi C_a$.
By construction the functor $_\phi r_{ea}$ is full and $_\phi \efC$ is a subpresheaf of $\efC$.

The functors $_\phi r_{ea}$ may be not bijective on the objects, 
so we define ${}_\beta C_a$, $a \in K$, as the \sC category
\[
\obj {}_\beta C_a := \obj \wt \efC \times \{ a \} \ni \varrho^a \equiv (\varrho,a)
\ \ , \ \
{}_\beta(\varrho^a,\varsigma^a) := \ _\phi (\varrho_a,\varsigma_a) \ .
\]
Then we add the presheaf structure
\[
_\beta r_{aa'} : {}_\beta C_{a'} \to {}_\beta C_a
\ \ , \ \
_\beta r_{aa'}\varrho^{a'} := \varrho^a
\ \ , \ \
_\beta r_{aa'}(f) := r_{aa'}(f)
\ \ , \ \
\forall a \leq a'
\ , \
f \in {}_\beta(\varrho^{a'},\varsigma^{a'}) 
\ .
\]
By definition ${}_\beta r_{aa'}$ is bijective on the objects, and any
\[
_\beta r_{aa'} : {}_\beta(\varrho^{a'},\varsigma^{a'}) \to {}_\beta(\varrho^a,\varsigma^a) \ \ , \ \ a \leq a' \ ,
\]
is bijective because $r_{aa'}$ is bijective on
$_\phi (\varrho_a,\varsigma_a) = {}_\beta(\varrho^a,\varsigma^a)$.
This proves that ${}_\beta \efC$ is a presheaf bundle.
We now define the presheaf morphism
\[
P : {}_\beta \efC \to \ \efC
\ \ , \ \
P_a(\varrho^a) := \varrho_a
\ , \
P_a(f) := f
\ \ , \ \
\forall \varrho^a \in \obj _\beta C_a
\ , \ 
f \in {}_\beta(\varrho^a,\varsigma^a)
\ , \
a \in K
\ ,
\]
inducing the *-functor 
$\wt P : \wt{_\beta \efC} \to \wt \efC$.
Now, sections $\wt \varrho \in \obj \wt{_\beta \efC}$ are of the type 
$\wt \varrho_a = \varrho^a$, $\forall a \in K$,
for some fixed $\varrho \in \obj \wt \efC$, so $\wt P$ is bijective on the objects.
Passing to arrows, we note that if 
$t \in (\varrho,\varsigma)$, $\varrho,\varsigma \in \obj \wt \efC$,
then $t_a \in {}_\phi(\varrho_a,\varsigma_a) = ( \varrho , \varsigma )_a$ for all $a \in K$,
so $\wt P$ is an isomorphism as desired. 
\end{proof}

We add further structure. 
We say that a presheaf $\efT = (T,r)^K$ of tensor \sC categories $(T_o)_{\otimes_o}$, $o \in K$, 
is a \emph{tensor presheaf} whenever
\[
r_{oo'} \circ \otimes_{o'} \ = \ \otimes_o \circ (r_{oo'} \times r_{oo'})
\ \ , \ \
\forall o \leq o'
\ .
\]
We say that $\efT$ has \emph{simple units} whenever $(\iota_o,\iota_o) \simeq \bC$ for all $o \in K$, where $\iota_o \in \obj T_o$ is the identity object of $T_o$.
The tensor structure shall be emphasized with the notation $\efT_\otimes$ and morphisms $\phi : \efT \to \efT'$ such that 
$\phi_o \circ \otimes_o = \otimes_o' \circ (\phi_o \times \phi_o)$, $o \in K$,
are denoted by $\phi : \efT_\otimes \to \efT'_{\otimes'}$.

\

The following example is motivated by the analysis of superselection structures in algebraic quantum field theory
on generic spacetimes.
\begin{ex}{\bf Nets of tensor categories defined by nets of \sC algebras, \cite[\S 27]{Rob}.}
\label{ex.Rob}
{\it
Let $\efR = (R,\jmath)_K$ be a net of \sC algebras on the Hilbert space $H$,
that is, any $R_o \subset B(H)$ is a unital \sC algebra and any $\jmath_{o'o}$, $o \leq o'$,
is the inclusion map. For any $a \in K$ we consider the \sC algebra 
\[
R^a \, := \, C^* \{ R_o  : o > a \} \, \subseteq B(H)
\]
(note that $R_a \subseteq R^a$)
and, given $\rho,\sigma \in {\bf end}R^a$, the space $(\rho,\sigma) \subseteq R^a$ 
defined as in (\ref{eq.intro0}). We define the tensor category
\[
T_a := 
\left\{
\begin{array}{ll}
\obj T_a \ := \ \{ \rho \in {\bf end}R^a : \rho(R_o) \subseteq R_o , \forall o > a \} \ ,
\\
{\bf arr}T_a \ := \ \{ (\rho,\sigma)_a := R_a \cap (\rho,\sigma) \ , \ \forall \rho,\sigma \in \obj T_a  \} \ ,
\end{array}
\right.
\]
endowed with the tensor product (\ref{eq.intro1}).
If $a \leq a'$ then $R^{a'} \subseteq R^a$ and any $\rho \in \obj T_a$ restricts
to an endomorphism $\jmath_{a'a}\rho \in \obj T_{a'}$. Moreover there are obvious inclusions
\[
\jmath_{a'a} : (\rho,\sigma)_a \to ( \jmath_{a'a}\rho , \jmath_{a'a}\sigma )_{a'}
\ \ , \ \
\rho,\sigma \in \obj T_a \ ,
\]
and this yields the net $\efT = (T,\jmath)_K$.
It is easily seen that $\jmath$ preserves the tensor structure, so $\efT_\otimes$ is
a tensor net. The identity of $T_a$ is the identity automorphism $\iota_a \in {\bf end}R^a$ and
$(\iota_a,\iota_a)_a = R_a \cap (R^a)'$.
}
\end{ex}

We say that the tensor presheaf $\efT_\otimes$ is \emph{symmetric} whenever 
any $(T_o)_{\otimes_o}$ has symmetry $\vareps_o$, and
\begin{equation}
\label{eq.NET1}
\vareps_o( r_{oo'}\varrho_{o'} , r_{oo'}\varsigma_{o'} )
\ = \
r_{oo'} (\vareps_{o'}(\varrho_{o'},\varsigma_{o'}))
\ \ , \ \qquad
\forall o \leq o'
\ , \
\varrho_o \in \obj T_o \ , \ \varsigma_o \in \obj T_o \ .
\end{equation}
The symmetry structure is emphasized with the notation $\efT_{\otimes,\vareps}$. Morphisms $\phi : \efT_\otimes \to \efT'_{\otimes'}$ such that $\efT'$ has symmetry $\vareps'$ and 
$\phi_o(\vareps(\varrho_o,\varsigma_o)) = \vareps'( \phi_o \varrho_o , \phi_o \varsigma_o)$, 
for all
$o \in K$,$\varrho_o , \varsigma_o \in \obj T_o$,
are denoted by
$\phi : \efT_{\otimes,\vareps} \to \efT'_{\otimes',\vareps'}$.

\begin{rem}
\label{rem.NET1}
Let us define, for any $\varrho , \varsigma \in \obj \wt \efT$,
\[
\vareps_{\varrho \varsigma , o} := \vareps_o(\varrho_o,\varsigma_o)
\ \ , \ \
\forall o \in K
\ .
\]
Then (\ref{eq.NET1}) implies that $\vareps_{\varrho \varsigma} \in {\bf arr} \ \wt \efT$
and this makes $\wt \efT$ a symmetric tensor \sC category (endowed with the obvious tensor structure).
Prop.\ref{prop.NET1} applies with the further property that ${}_\beta \efT_{\otimes,\vareps}$ 
is a symmetric tensor presheaf bundle.
\end{rem}

Applying Theorem \ref{thm_NET00} to the category ${\bf T}$ with objects 
small symmetric tensor \sC categories and arrows symmetric tensor *-functors, we obtain an equivalence
${\bf hom}( \pi_1(K) , {\bf T} ) \simeq {\bf bun}^\uparrow(K,{\bf T})$
from which it follows, given the symmetric tensor category $T_{\otimes,\eps}$, the one-to-one correspondence
\begin{equation}
\label{eq_LTGA1}
\unl{\hom}( \pi_1(K) , {\bf aut}T_{\otimes,\eps} )
\ \simeq \
\unl{\bf bun}^\uparrow(K,T_{\otimes,\eps})
\ .
\end{equation}

\

\noindent {\bf Hilbert presheaves} are, by definition, symmetric tensor full presheaves with fibres 
\emph{full} symmetric tensor subcategories of ${\bf Hilb}$
{\footnote{Note that we assume, in particular, that all the fibres of a Hilbert presheaf are endowed with the
           standard symmetry given by the flip operators 
           $\vartheta_{h \otimes h'}(v \otimes v') := v' \otimes v$, $v \in h$, $v' \in h'$.}}.
If $\efC_{\otimes,\vartheta} = (C,R,\otimes,\vartheta)^K$ is a Hilbert presheaf then any 
$h \in \obj C_a$, $a \in K$,
is a Hilbert space and, in particular, we have the identity object $\iota_a \equiv \bC$. 
We have canonical, bijective linear maps
\[
h \to (\iota_a,h)
\ \ , \ \qquad
v \mapsto v_\bullet \ : \ v_\bullet(\lambda) := \lambda  v \ , \ v \in H \, , \, \lambda \in \bC \ ,
\]
with inverses
$(\iota_a,h) \to h$, $f \mapsto f_\bullet := f(1)$ (so $v_{\bullet \bullet} = v$ for all $v \in h$).
Let $\varkappa \in \obj \wt \efC$. Since $\efC$ is full, for any $a \leq a'$ the functor $R_{aa'}$
defines a bijective, isometric linear map
$R_{aa'} : ( \iota_{a'} , \varkappa_{a'} ) \to ( \iota_a , \varkappa_a )$,
so we have the unitary operator
\begin{equation}
\label{eq_LTGA2}
R^\varkappa_{aa'} : \varkappa_{a'} \to \varkappa_a
\ \ : \ \
R^\varkappa_{aa'}v' := (R_{aa'}(v'_\bullet))_\bullet
\ , \ 
\forall v' \in \varkappa_{a'}
\ .
\end{equation}
Since $R$ fulfils the presheaf relations, the family $R^\varkappa := \{ R^\varkappa_{aa'} \}$ 
fulfils the presheaf relations. For convenience we set 
$R^\varkappa_{a'a} := (R^\varkappa_{aa'})^{-1}$, $\forall a \leq a'$,
so $R^\varkappa_{a''a'} \circ R^\varkappa_{a'a} = R^\varkappa_{a''a}$, $\forall a \leq a' \leq a''$,
and
$\efH_\varkappa = (\varkappa,R^\varkappa)_K$
is a Hilbert net bundle.

Let now $\varkappa , \varkappa' \in \obj \wt \efC$ and $a \leq a'$.
Since $R$ preserves tensor products, by (\ref{eq_LTGA2}) we find
\begin{equation}
\label{eq_LTGAa}
R^{\varkappa \otimes \varkappa'}_{aa'}
\ = \
R^\varkappa_{aa'} \otimes R^{\varkappa'}_{aa'}
\ .
\end{equation}
Let $f' \in (\varkappa_{a'} , \varkappa'_{a'})$ (so $f'$ is a linear map from $\varkappa_{a'}$ to $\varkappa'_{a'}$);
then for any $v \in \varkappa_a$ we find
\[
\begin{array}{lll}
\{ R_{aa'}(f') \} v & = 
\{ R_{aa'}(f') \} ( R_{aa'}( R^\varkappa_{a'a}v )_\bullet )  =  \\ & =
R_{aa'} ( f' \circ ( R^\varkappa_{a'a}v )_\bullet ) =
R_{aa'} ( \{ f' \circ R^\varkappa_{a'a} \} v )_\bullet =
\{ R^{\varkappa'}_{aa'} \circ f' \circ R^\varkappa_{a'a} \} v
\ ,
\end{array}
\]
that is,
\begin{equation}
\label{eq.LTGA2}
R_{aa'}(f') \ = \ R^{\varkappa'}_{aa'} \circ f' \circ R^\varkappa_{a'a}
\ \ , \ \
\forall a \leq a' \ .
\end{equation}
If $t \in (\varkappa,\varkappa')$ then 
\[
t_a \ = \ 
R_{aa'}(t_{a'}) \ \stackrel{ (\ref{eq.LTGA2}) }{=} \ 
R^{\varkappa'}_{aa'} \circ t_{a'} \circ R^\varkappa_{a'a}
\ \ , \ \
\forall a \leq a' \ .
\]
The previous equalities show that $t \in (\varkappa,\varkappa')$ if, and only if, 
$t \in ( \efH_\varkappa , \efH_{\varkappa'} )$,
and we conclude that
\begin{equation}
\label{lem_LTGA2}
\wt \efC \to {\bf bun}(K,{\bf Hilb})
\ \ , \ \
\varkappa \mapsto \efH_\varkappa 
\ \ , \ \
t \mapsto t
\ ,
\end{equation}
is an embedding. Simple computations show that (\ref{lem_LTGA2}) preserves tensor product and symmetry.

\section{Nets of \sC algebras and group duals.}
\label{sec_end}

In the present section we focus on \sC net bundles (\cite{BR08,RV11}).
In particular, we give examples of group actions $\alpha : G \to {\bf aut}F$ such that 
there are \sC net bundles $\efA$ with fibre $F^\alpha$ that do not uniquely determine, or do not determine at all,
a \sC net bundle with fibre $F$ having $\efA$ as fixed-point net.
The basic ingredient of our constructions is the Cuntz algebra (\cite{Cun77}).

\

\noindent {\bf Nets of \sC dynamical systems.} Let $(A,\rho)$ be a \sC dynamical system. 
Then we have a tensor \sC category $\wa \rho$ defined as the full tensor \sC subcategory
of ${\bf end}A$ with objects the tensor powers
%
%
$\rho^r := \rho \circ \ldots \circ \rho$, $r \in \bN$
(see (\ref{eq.intro1})).
We say that $\rho$ is symmetric whenever $\wa \rho$ has symmetry $\eps$ 
(see \S \ref{sec_back} or \cite[Eq.4.5-4.7]{DR87}),
and in this case we say that $(A,\rho,\eps)$ is a \emph{symmetric \sC dynamical system}.
Using elementary combinatorics one can prove that any 
$\eps_{\rho^r,\rho^s}$, $r,s \in \bN$,
is the product of elements of the type 
$\rho^s(\eps_{\rho^2,\rho^2})$, $s=0,1, \ldots$ .
Thus $\eps$ is determined by $\eps_{\rho^2,\rho^2} \in (\rho^2,\rho^2)$
(see \cite[Eq.2.6,Eq.4.8]{DR87}), and, with an abuse of notation, we write 
$\eps \equiv \eps_{\rho^2,\rho^2} \in UA$.
If $\alpha \in {\bf aut}_{\rho,\eps}A$ then,
for all $r,s \in \bN$ and $t \in (\rho^r,\rho^s)$, $t' \in (\rho^{r'},\rho^{s'})$, 
we have
\[
\left\{
\begin{array}{ll}
\alpha \circ \rho^r = \rho^r \circ \alpha \ ,
\\
\alpha(t) \rho^r(v) = 
\alpha(t \cdot \rho^r \circ \alpha^{-1}(v)) = 
\alpha( \rho^s \circ \alpha^{-1}(v) \cdot t) =
\rho^s(v) \cdot \alpha(t) 
\ \Rightarrow \
\alpha(t) \in (\rho^r,\rho^s)
\ ,
\\
\alpha(t \otimes t') = 
\alpha(t\rho^r(t')) = 
\alpha(t) \rho^r(\alpha(t')) =
\alpha(t) \otimes \alpha(t') \ ,
\\
\alpha \circ \rho^r (\eps) = \rho^r(\eps) 
\ \Rightarrow \
\alpha(\eps_{\rho^r,\rho^s}) = \eps_{\rho^r,\rho^s}
\ .
\end{array}
\right.
\]
The previous equalities say that $\alpha$ defines a symmetric tensor automorphism 
$\wa \alpha$ of $(\wa \rho,\otimes,\eps)$.
We have the following result, whose proof is given in \cite[Theorem 1]{Vas09}:
\begin{prop}
\label{prop_LTGA1a}
Let $A$ be a \sC algebra and $\rho \in {\bf end}A$ with symmetry $\eps$.
If the vector space generated by $\cup_{rs}(\rho^r,\rho^s)$ is dense in $A$
in the norm topology, then every $\beta \in {\bf aut}(\wa \rho,\otimes,\eps)$
is of the type $\beta=\wa \alpha$, $\alpha \in {\bf aut}_{\rho,\eps}A$.
\end{prop}

As a consequence of the previous proposition and (\ref{eq_LTGA1}), when $A$ is generated by 
$\cup_{rs}(\rho^r,\rho^s)$ 
we have the one-to-one correspondence
\begin{equation}
\label{eq_LTGA1a}
\unl{\bf bun}^\uparrow(K,\wa \rho_{\otimes,\eps})
\ \simeq \ 
\unl{\hom}( \pi_1(K) , {\bf aut}_{\rho,\eps}A )
\ .
\end{equation}
We now interpret the r.h.s. of (\ref{eq_LTGA1a}) in terms of \sC net bundles.
To this end, let $\efA = (A,\jmath)_K$ be a \sC net bundle and $\varrho \in {\bf end}\efA$;
then any $\varrho_a \in {\bf end}A_a$, $a \in K$, defines the tensor category 
$\wa \varrho_a$
of tensor powers of $\varrho_a$, and using the relations 
\begin{equation}
\label{eq.SE01}
\jmath_{a'a} \circ \varrho_a \circ \jmath_{aa'} = \varrho_{a'}
\ \ , \ \
\forall a \leq a'
\ ,
\end{equation}
we conclude that defining for all $a \leq a'$
\begin{equation}
\label{eq.SE01a}
\wa \jmath_{a'a}(\varrho_a^r) \, := \, \varrho_{a'}^r
\ , \
\wa \jmath_{a'a}(t) := \jmath_{a'a}(t)
\ \ , \ \qquad
\forall r,s \in \bN
\ , \
t \in (\varrho_a^r,\varrho_a^s)
\ ,
\end{equation}
yields tensor isomorphisms $\wa \jmath_{a'a} : \wa \varrho_a \to \wa \varrho_{a'}$ and the 
net bundle of tensor \sC categories
$(\wa \varrho,\wa \jmath,\otimes)_K$.
By (\ref{eq.SE01},\ref{eq.SE01a}), $\varrho$ is a section of $\wa \varrho$.
For the symmetry structure we prove the following result.
\begin{thm}
\label{thm_NET03}
Let $( A_* , \rho , \eps )$ be a pointed \sC dynamical system. Given a \sC net bundle $\efA = (A,j;A_*)_K$, 
the following are equivalent: 
(1) There are $\varrho \in {\bf end} \efA$ and $\vareps \in \wt \efA$ such that
$\varrho_a = \rho$, $\vareps_a = \eps$;
(2) The structure group of $\efA$ admits a reduction 
$\chi : \pi_1(K) \to {\bf aut}_{\rho , \eps } A_*$.
\end{thm}

\begin{proof}
(1) $\Ra$ (2): By (\ref{eq_FC01},\ref{eq_FC03},\ref{eq_FC04}), we have the map
\[
\dot{\chi}(\gamma) := 
\jmath_{\partial_0 b_n,|b_n|} \circ \jmath_{|b_n|,\partial_1 b_n}
\circ \cdots \circ
\jmath_{\partial_0 b_1,|b_1|} \circ \jmath_{|b_1|,\partial_1 b_1}
\ ,
\]
$\gamma := b_n * \cdots * b_1 \in K(a)$, which yields the holonomy representation of $\efA$. Since $\rho$ preserves the net structure, recalling that $\partial_1 b_1 = a$ we find
\[
\jmath_{|b_1|,\partial_1 b_1} \circ \rho 
= 
\rho_{|b_1|} \circ \jmath_{|b_1|,\partial_1 b_1}
\ \ , \ \
\jmath_{\partial_0 b_1,|b_1|} \circ \rho_{|b_1|}
=
\rho_{\partial_1 b_2} \circ \jmath_{\partial_0 b_1,|b_1|}
\ .
\]
Iterating the above identities for all the 1--simplices of $\gamma$, we conclude that $\dot{\chi} (\gamma) \circ \rho = \rho \circ \dot{\chi} (\gamma)$. In the same way, if $\vareps \in \wt \efA$ is such that $\vareps_a = \eps$, then
\[
\jmath_{|b_1|,\partial_1 b_1} (\eps) = \eps (|b_1|)
\ \ , \ \
\jmath_{\partial_0 b_1,|b_1|} \circ \eps (|b_1|) = \eps (\partial_1 b_2)
\ ,
\]
and we conclude that
$\{ \dot{\chi} (\gamma) \}  (\eps) = \eps$.
(2) $\Ra$ (1): We make the identification $A_* \equiv A_a$, so $\rho \in {\bf end}A_a$, $\eps \in A_a$. By hypothesis, for any path $\gamma := b_n * \cdots * b_1 \in K(a)$ we find
\begin{equation}
\label{eq_thm_NET03B}
\dot{\chi} (\gamma) \circ \rho = \rho \circ \dot{\chi} (\gamma)
\ \ , \ \
\{ \dot{\chi} (\gamma) \} (\eps) = \eps
\ ,
\end{equation}
where $\dot{\chi}(\gamma) \in {\bf aut}A_*$ is defined by (\ref{eq_FC04}). 
So, if $\gamma_1 , \gamma_2 \in K(a,o)$, then 
$\dot{\chi} (\gamma_2^{-1}*\gamma_1) \circ \rho 
= 
\rho \circ \dot{\chi} (\gamma_2^{-1}*\gamma_1)$
and
$\{ \dot{\chi} (\gamma_2^{-1}*\gamma_1) \} (\eps) = \eps$.
This implies, recalling (\ref{eq_FC04A}), that
\[
Z(\gamma_1) \circ \rho \circ Z(\gamma_1)^{-1}
= 
Z(\gamma_2) \circ \rho \circ Z(\gamma_2)^{-1}
\ , \
\{ Z(\gamma_1) \} (\eps) = \{ Z(\gamma_2) \} (\eps)
\ .
\]
The previous equalities imply that the following maps are well-defined, as they do not depend on the choice of $\gamma : a \to o$:
\[
\varrho_o := Z(\gamma) \circ \rho \circ Z(\gamma)^{-1}
\ \ , \ \
\vareps_o := \{ Z(\gamma) \} (\eps)
\ \ , \ \
o \in K
\ .
\]
Now, for each $o \leq o'$ we compute
\[
\varrho_{o'} \circ \jmath_{o'o} = 
Z(\gamma) \circ \rho \circ Z(\gamma)^{-1} \circ \jmath_{o'o} =
\jmath_{o'o} \circ Z(\gamma') \circ \rho \circ Z(\gamma')^{-1} =
\jmath_{o'o} \circ \varrho_o
\ ,
\]
where $\gamma' := \gamma * b_{o'o} \in K(a,o')$ (the notation $b_{o'o}$ is the one used in (\ref{eq_nets02})). This proves that $\varrho := \{ \varrho_o \}$ is a well-defined endomorphism of $\efA$. 
In the same way we prove that $\vareps \in \wt \efA$.
\end{proof}

Applying the previous theorem with $\vareps_o \equiv 1$ we have that there is 
$\varrho \in {\bf end}\efA$ 
if, and only if, $\efA$ has holonomy in ${\bf aut}_{\rho}A_*$.
We call \emph{pointed dynamical \sC net bundle} the triple $(\efA ,\varrho,\vareps)$; 
to be concise, sometimes in the sequel we will write $\efA_{\varrho,\vareps}$. 
There is an obvious notion of morphism of pointed dynamical \sC net bundles:
\[
\eta : \efA_{\varrho,\vareps} \to \efA'_{\varrho',\vareps'} 
\ \Leftrightarrow \ 
\eta \in (\efA,\efA')
\ , \
\eta_o \circ \varrho_o = \varrho'_o \circ \eta_o
\ \ , \ \
\eta_o (\vareps_o) = \vareps'_o
\ \ , \ \
\forall o \in K
\ .
\]
We denote the set of pointed dynamical \sC net bundles with fibre $(A_*,\rho,\eps )$ by 
${\bf bun}(K;A_*,\rho,\eps)$ 
and the  associated set of isomorphism classes by $\unl{\bf bun}(K;A_*,\rho,\eps)$. 
We can now return on the question of symmetry of $(\wa \varrho,\wa \jmath,\otimes)_K$. 
\begin{cor}
\label{cor_NET03}
Let $(A_*,\rho,\eps)$ be a symmetric \sC dynamical system such that $A_*$ is generated by 
$\cup_{rs}(\rho^r,\rho^s)$.
Then for any 
$\efA_{\varrho,\vareps} \in {\bf bun}(K;A_*,\rho,\eps)$
the tensor net bundle 
$(\wa \varrho,\wa \jmath,\otimes,\vareps)_K$ 
is symmetric and has holonomy coinciding with the one of $\efA_{\varrho,\vareps}$,
having used the identification of Prop.\ref{prop_LTGA1a}.
There are one-to-one correspondences
\begin{equation}
\label{eq_NET03A}
\unl{\hom}( \pi_1(K),{\bf aut}_{\rho,\eps}A_*) 
\ \simeq \
\unl{\bf bun}(K;A_*,\rho,\eps)
\ \simeq \
\unl{\bf bun}^\uparrow(K,\wa \rho_{\otimes,\eps})
\ .
\end{equation}
\end{cor}

\begin{proof}
The first bijection in (\ref{eq_NET03A}) is a consequence of Theorem \ref{thm_NET03},
so we prove that there is a bijection
$\unl{\hom}( \pi_1(K),{\bf aut}_{\rho,\eps}A_*) \simeq 
 \unl{\bf bun}^\uparrow(K,\wa \rho_{\otimes,\eps})$.
To this end, we note that if 
$\chi : \pi_1(K) \to {\bf aut}_{\rho,\eps}A_*$ 
is the holonomy of $\efA_{\varrho,\vareps}$ then we have 
$\{ \chi(p) \}(\eps) = \eps$ for all $p \in \pi_1(K)$.
So, considering the induced *-isomorphisms
$Z(\gamma) : A_a \to A_o$, $\gamma : a \to o$
(see (\ref{eq_FC04A})), we define
$\vareps_o := \{ Z(\gamma) \}(\eps)$, $\forall o \in K$.
Using (\ref{eq.SE01}) we conclude that $\vareps$ is a section of $\efA$ and that 
$\vareps_o$ fulfils (\ref{eq.eps1},\ref{eq.eps2}) for any $o \in K$, 
so we have the desired symmetric tensor net bundle 
$(\wa \varrho,\wa \jmath,\otimes,\vareps)_K$.
On the other side, if 
$(\wa \varrho,\wa \jmath,\otimes,\vareps)_K \in {\bf bun}^\uparrow(K,\wa \rho_{\otimes,\eps})$
then (\ref{eq_LTGA1a}) implies that $\wa \varrho$ has holonomy
$\chi : \pi_1(K) \to {\bf aut}_{\rho,\eps}A_*$,
which determines the net bundle $\efA_{\varrho,\vareps}$.
\end{proof}

\noindent {\bf Cuntz algebras.} 
We call \emph{symmetric \sC net bundle} a pointed dynamical \sC net bundle 
of the type considered in the previous corollary.
In the following paragraphs we study a "universal" class of symmetric \sC net bundles,
defined by means of the Cuntz algebras $\mO_d$, $d \in \bN$. 
%
%
We start recalling that $\mO_d$ is the universal \sC algebra generated by $d$ orthogonal
isometries $\psi_1 , \ldots , \psi_d$ with total support the identity, that is
\[
\psi_h^* \psi_k = \delta_{hk}1
\ \ , \ \
\sum_k \psi_k \psi_k^* = 1
\ ,
\]
where $\delta$ is the Kronecker symbol (\cite{Cun77}). $\mO_d$ has a well-known structure
of symmetric \sC dynamical system,
\[
\sigma \in {\bf end}\mO_d
\ , \
\sigma(t) \, := \, \sum_k \psi_k t \psi_k^*
\ , \
\forall t \in \mO_d
\ \ \ , \ \ \
\theta \, := \, \sum_{hk} \psi_h \psi_k \psi_h^* \psi_k^*  \ \in ( \sigma^2,\sigma^2 )
\]
(see \cite[\S 2]{DR87}), and is endowed with the action
\begin{equation}
\label{eq_FL04}
\ud \to {\bf aut} \mO_d 
\ \ , \ \
u \mapsto \wa u \ : \ \wa u (\psi_h) := \sum_k u_{hk} \psi_k
\ ,
\end{equation}
where $u_{hk} \in \bC$, $h,k= 1 , \ldots , d$, are the matrix elements of $u$ (\cite[\S 1]{DR87}). 
Clearly, the action (\ref{eq_FL04}) restricts to any subgroup $G \subseteq \ud$. Since 
$\wa u \circ \sigma = \sigma \circ \wa u$ and $\wa u (\theta) = \theta$, $\forall u \in \ud$, 
we have that the fixed point algebra $\mO_G \subseteq \mO_d$ defines the symmetric \sC dynamical system 
\[
(\mO_G,\rho,\theta) \ \ , \ \ \rho := \sigma |_{\mO_G} \ .
\]
%
%
Let now $NG \subseteq \ud$ denote the normalizer of $G$ in $\ud$; we consider the quotient
\[
q : NG \to QG := NG/G
\]
and define
${\bf aut}^G\mO_d := \{ \alpha \in {\bf aut}_{\sigma,\theta}\mO_d \ : \ \alpha |_{\mO_G} \in {\bf aut}_{\rho,\theta}\mO_G
\}$.
\begin{lem}
\label{lem_NG}
Let $u \in \ud$. Then $u \in NG$ if, and only if, $\wa u \in {\bf aut}^G\mO_d$. If $v \in NG$, then $\wa u |_{\mO_G} = \wa v |_{\mO_G}$ if, and only if, $q(u) = q(v)$.
\end{lem}

\begin{proof}
If $u \in NG$ then $g' := ugu^* \in G$ for all $g \in G$ and 
$\wa u (t) = \wa u \circ \wa g (t) = \wa u (t) = \wa g' \circ \wa u (t)$ 
for all $t \in \mO_G$. Since every $g' \in G$ is of the type $g' = ugu^*$ for some $g \in G$, we conclude that $\wa u (t) \in \mO_G$ and $\wa u \in {\bf aut}^G\mO_d$.
On the converse, if $\wa u \in {\bf aut}^G\mO_d$ then $\wa u^{-1}(t) \in \mO_G$ for all $t \in \mO_G$ and, defining $g' := ugu^*$ for each $g \in G$, we find 
$\wa g' (t) = \wa u \circ \wa g \circ \wa u^{-1}(t) = \wa u \circ \wa u^{-1}(t) = t$.
Thus we conclude that $\wa g'$ is in the stabilizer of $\mO_G$ in ${\bf aut}\mO_d$, so that, by \cite[Corollary 3.3]{DR87}, $g' \in G$ i.e. $u \in NG$. 
Finally, again by \cite[Corollary 3.3]{DR87} we have that $\wa u |_{\mO_G} = \wa v |_{\mO_G}$ if, and only if, $uv^* \in G$, i.e. $q(u) = q(v)$.
\end{proof}

Thus the map (\ref{eq_FL04}) defines a monomorphism $NG \to {\bf aut}^G\mO_d$, which yields the monomorphism
\begin{equation}
\label{eq_FL04A}
QG \to {\bf aut}_{\rho,\theta}\mO_G 
\ \ , \ \
y \mapsto \wa y
\ .
\end{equation}
By \cite[Def.4-Theorem 7]{Vas09}, the map (\ref{eq_FL04A}) is an isomorphism when $G \subseteq \sud$ or $G = \bT \subseteq \ud$
acts on $\bC^d$ by scalar multiplication; in particular, for $G = \{ 1 \}$ we have the isomorphism
\begin{equation}
\label{eq_prop_FL01}
\ud \to {\bf aut}_{\sigma,\theta}\mO_d
\ \ , \ \
u \mapsto \wa u
\ .
\end{equation}
To simplify the proofs of the following results we give a categorical version of (\ref{eq_prop_FL01}). 
Let ${\bf Hilb}^{-1}$ denote the category with objects finite dimensional Hilbert spaces 
and arrows the sets $(H,H')^{-1}$ of \emph{unitary} operators from $H$ to $H'$; by \cite[\S 4]{DR89}, 
there is a *-functor
\begin{equation}
\label{eq_cuntz_dr}
{\bf Hilb}^{-1} \to {\Calg}
\ \ , \ \
H \mapsto \mO_H
\ \ , \ \
u \mapsto \wa u
\ ,
\end{equation}
where $\mO_H$ is isomorphic to the Cuntz algebra of order the dimension of $H$, and 
$\wa u : \mO_H \to \mO_{H'}$ 
is the isomorphism defined by
$\wa u(\psi_h) := \sum_k u_{hk} \psi'_k$,
where $\{ \psi_h \}$, $\{ \psi'_k \}$ are the sets of isometries generating $\mO_H$, $\mO_{H'}$ respectively. 
By the previous considerations, if we denote the canonical endomorphism of $\mO_H$ 
and the symmetry operator respectively by $\sigma_H$ and $\theta_H$, then
$\wa u \circ \sigma_H = \sigma_{H'} \circ \wa u$
and
$\wa u (\theta_H) = \theta_{H'}$,
so (\ref{eq_cuntz_dr}) takes values in the category of symmetric \sC dynamical systems.

Given $G \subseteq \ud$, we consider the defining representation 
\[
\pi_G : G \to \ud
\]
and the tensor \sC category $\wa \pi_G$ with objects the tensor powers $\pi_G^r$, $r \in \bN$, 
and arrows the intertwiner spaces. $\wa \pi_G$ has the symmetry $\varphi$ inherited from ${\bf Hilb}$
(see \cite[\S 1]{DR87}), so we write 
$\wa \pi_{G;\otimes,\varphi}$
to emphasize the symmetric tensor structure.
Let now $u \in \ud$. Then $u$ defines a symmetric tensor automorphism of the category of tensor powers
of $H := \bC^d$,
\begin{equation}
\label{eq.NGa}
\wa u(H^r) := H^r
\ \ , \ \
\wa u(t) := u^s \circ t \circ u^{r,*}
\ \ , \ \
\forall r,s \in \bN
\ , \
t \in (H^r,H^s)
\ ,
\end{equation}
where $u^r \in UH^r$ is the $r$-fold tensor power.
Since the only unitaries inducing the identity automorphism of $\wa \pi_G$ are the elements of $G$ (\cite[Corollary 3.3]{DR87}), 
reasoning as in Lemma \ref{lem_NG} we have
\begin{equation}
\label{eq.NG}
\left\{
\begin{array}{ll}
u \in NG \ \ \Leftrightarrow \ \ \wa u(t) \in (\pi_G^r,\pi_G^s) \ , \ \forall t \in (\pi_G^r,\pi_G^s) \ ,
\\
\wa u |_{\wa \pi_G} \, = \, \wa v |_{\wa \pi_G} \ \ \Leftrightarrow \ \ q(u) = q(v) \ , \ \forall u,v \in NG \ .
\end{array}
\right.
\end{equation}
\begin{thm}
\label{thm_OG}
Let $G \subseteq \ud$. Then there are maps
\begin{equation}
\label{eq_thm_FL08}
\left\{
\begin{array}{ll}
{\hom}( \pi_1(K) , QG ) \ \stackrel{(\bullet)}{\to} \ {\bf bun}(K ; \mO_G , \rho , \theta ) 
\ \ , \ \
\chi \mapsto \efA_{\chi ; \varrho,\vartheta} \ ,
\\
{\hom}( \pi_1(K) , QG )				\ \stackrel{(\bullet \bullet)}{\to} \ 
{\bf bun}^\uparrow(K,\wa \rho_{\otimes,\theta}) \ \stackrel{(\bullet \bullet \bullet)}{\to} \ 
{\bf bun}^\uparrow(K,\wa \pi_{G;\otimes,\varphi}) \ .
\end{array}
\right.
\end{equation}
The map $(\bullet \bullet \bullet)$ is always an isomorphism, whilst $(\bullet)$, $(\bullet \bullet)$ are isomorphisms
when $G \subseteq \sud$ or $G = \bT \subseteq \ud$.
\end{thm}

\begin{proof}
(\ref{eq_thm_FL08}.1) is a direct consequence of (\ref{eq_NET03A}) and the fact that (\ref{eq_FL04A}) is an isomorphism when $G \subseteq \sud$ or $G = \bT \subseteq \ud$ (\cite[Def.4-Theorem 7]{Vas09}). 
About (\ref{eq_thm_FL08}.2), we note that $\mO_G$ is generated by $\cup_{rs}(\rho^r,\rho^s)$
(see \cite[\S 1,\S 3]{DR87}), and that there is an isomorphism 
$\wa \rho_{\otimes,\theta} \simeq \wa \pi_{G;\otimes,\varphi}$
(see \cite[Theorem 3.5]{DR87}), so the proof follows by (\ref{eq_thm_FL08}.1) and Corollary \ref{cor_NET03}.
\end{proof}

\begin{prop}
\label{prop_FL01}
Let $d \in \bN$ and $\efH = (H,j ; \bC^d )_K$ a Hilbert net bundle. 
Then $\efH$ defines a symmetric \sC net bundle
\[
\mO_\efH = ( \mO_H , \wa j ; \mO_d )_K
\ \ , \ \
( \mO_\efH , \varsigma , \vartheta ) \in {\bf bun}(K;\mO_d,\sigma,\theta)
\ .
\]
The symmetric tensor net bundle $(\wa \varsigma,\wa \jmath,\otimes,\vartheta)_K$
has fibre the full subcategory of ${\bf Hilb}$ of tensor powers of $\bC^d$.
Moreover, $\efH , \efH'$ are isomorphic if and only if the there is an isomorphism 
$\mO_{\efH ; \varsigma , \vartheta} \simeq \mO_{\efH' ; \varsigma' , \vartheta'}$.
\end{prop}

\begin{proof}
If $\chi : \pi_1(K) \to \ud$ is the holonomy of $\efH$, then composing $\chi$ with (\ref{eq_prop_FL01})
we obtain the desired $\mO_{\efH ; \varsigma , \vartheta}$.
Since (\ref{eq_prop_FL01}) is an isomorphism we conclude that $\efH , \efH'$ are isomorphic
if and only if $\mO_{\efH ; \varsigma , \vartheta}$, $\mO_{\efH' ; \varsigma' , \vartheta'}$
are isomorphic. 
The assertion about $(\wa \varsigma,\wa \jmath,\otimes,\vartheta)_K$ follows by applying
Theorem \ref{thm_OG} to the trivial group $G = \{ 1 \}$.
\end{proof}

Let ${\bf U}\efH = (UH,\jmath_* ; \ud )_K$ denote the net bundle of unitary automorphisms of $\efH$. 
The family 
$SUH := \{ u \in UH_a : \det u = 1  \}_a$ 
is stable under $\jmath_*$, thus we have a group net subbundle of ${\bf U}\efH$ that we denote by
${\bf SU} \efH \in {\bf bun}(K,\sud)$.
Let $\efG = (Y,i;G)_K$ be a group net bundle and $(\efH,\alpha)$ a $\efG$-Hilbert net bundle; then we may regard $\alpha$ as a morphism $\alpha \in (\efG,{\bf U}\efH)$, and the equalities
\[
\jmath_{*,o' o} \circ \alpha_o (Y_o) = \alpha_{o'} \circ i_{o'o}(Y_o) = \alpha_{o'}(Y_{o'})
\ \ , \ \
o \leq o'
\ ,
\]
imply that the set $\alpha(Y) \subseteq UH$ is stable under $\jmath_*$. Thus any gauge action on $\efH$ defines a group net subbundle $\efG_\alpha$ of ${\bf U}\efH$. In the sequel we consider exclusively gauge actions induced by net subbundles of ${\bf U}\efH$.
An immediate consequence of Lemma \ref{lem_LTGA1} is the following:
\begin{lem}
\label{lem_FL07}
Let $\efH \in {\bf bun}(K,\bC^d)$. Group net subbundles of ${\bf U}\efH$ with fibre $G \subseteq \ud$ are in one-to-one correspondence with reductions to $NG$ of the structure group of $\efH$.
\end{lem}

\begin{prop}
\label{prop_FL07}
Let $\efH = (H,\jmath)_K \in {\bf bun}(K,\bC^d)$ a Hilbert net bundle and 
\[
\efG \subseteq {\bf U}\efH \ \ , \ \ \efG = (Y,\jmath_*)_K \ ,
\]
a group net subbundle. Then there is a gauge action 
$\efG \times_K \mO_\efH \to \mO_\efH$
and a closed group $G \subseteq \ud$ such that the fixed-point net bundle 
$\mO_\efH^\efG \subseteq \mO_\efH$ 
yields the symmetric \sC net bundle 
\[
\mO_{\efH ; \varrho,\vartheta}^\efG \in {\bf bun}(K;\mO_G,\rho,\theta) \ .
\]
If $\chi : \pi_1(K) \to NG$ is the holonomy representation of $\efH$ defined in Lemma \ref{lem_FL07}, 
then $\mO_\efH^\efG$ has holonomy representation induced by
\[
q \circ \chi : \pi_1(K) \to QG \ .
\]
Finally, $q \circ \chi$ is the holonomy of the symmetric tensor net bundle
$(\wa \varrho,\wa \jmath,\otimes,\vartheta)_K \in {\bf bun}^\uparrow(K,\wa \pi_{G;\otimes,\varphi})$,
which is a subbundle of $(\wa \varsigma,\wa \jmath,\otimes,\vartheta)_K$.
\end{prop}

\begin{proof}
Writing (\ref{eq_GA00}) for the action $\efG \times_K \efH \to \efH$ we have
$\jmath_{a'a} \circ y = j_{* ; a'a}(y) \circ \jmath_{a'a}$,
$\forall y \in Y_a$,
$a \leq a'$,
so the desired gauge action is defined by composition with the functor (\ref{eq_cuntz_dr}),
i.e.
\[
\wa \jmath_{a'a} \circ \wa y = \wa \jmath_{* ; a'a}(y) \circ \wa \jmath_{a'a}
\ ,
\]
with
$\wa \jmath_{a'a} : \mO_{H,a} \to \mO_{H,a'}$,
$\wa y \in {\bf aut}\mO_{H,a}$, 
and
$\wa \jmath_{* ; a'a}(y) \in {\bf aut}\mO_{H,a'}$. 
Since all the above maps are isomorphisms of symmetric \sC dynamical systems, 
the fixed-point net 
$\mO_\efH^\efG = ( \mO_Y,\wa \jmath)_K$
acquires the structure of dynamical \sC net bundle
$( \mO_\efH^\efG , \varrho , \vartheta )$, 
$\varrho_a := \varsigma_a|_{\mO_{Y,a}} \in {\bf end}\mO_{Y,a}$, $a \in K$.
Fixing a standard fibre $G := Y_a$, $a \in K$, yields 
$\mO_{\efH ; \varrho,\vartheta}^\efG \in {\bf bun}(K;\mO_G,\rho,\theta)$
as desired.
Now, by Lemma \ref{lem_FL07} $\efG$ yields a reduction 
$\chi : \pi_1(K) \to NG$
of the structure group of $\efH$ such that 
$\efG$ has holonomy 
$\ad \chi : \pi_1(K) \to {\bf aut}G$
and $\mO_\efH$ has holonomy 
$\wa \chi : \pi_1(K) \to {\bf aut}_{\sigma,\theta}\mO_d$ 
(obtained by composition with (\ref{eq_cuntz_dr})). By (\ref{eq_FL04A}), any automorphism 
\[
\{ \wa \chi(p) \} |_{\mO_G} \in {\bf aut}_{\rho,\theta}\mO_G 
\ \ , \ \ 
p \in \pi_1(K)
\ ,
\]
is determined by the coset $q \circ \chi(p) \in QG$, 
so $\mO_{\efH ; \varrho,\vartheta}^\efG$ has holonomy $q \circ \chi$.
Finally, $(\wa \varrho,\wa \jmath,\otimes,\vartheta)_K$ has holonomy
$q \circ \chi$ by Corollary \ref{cor_NET03}, and since we have the inclusion 
$\mO_{\efH ; \varrho,\vartheta}^\efG \subseteq \mO_{\efH ; \varsigma,\vartheta}$
we also find that $(\wa \varrho  ,\wa \jmath,\otimes,\vartheta)_K$
is a subbundle of $(\wa \varsigma,\wa \jmath,\otimes,\vartheta)_K$.
\end{proof}

It is natural to ask whether every element of 
${\bf bun}(K;\mO_G,\rho,\theta)$
is of the type
$\mO_{\efH ; \varrho,\vartheta}^\efG$
for some Hilbert net bundle $\efH$ and $\efG \subseteq {\bf U}\efH$. 
In the following result we show that this question is equivalent to a lifting problem. 
\begin{thm}
\label{thm_FL08}
Let $G \subseteq \ud$, $\chi \in \hom(\pi_1(K),QG)$ and 
\[
( \efA_\chi , \tau, \vareps ) \in {\bf bun}(K;\mO_G,\rho,\theta)
\]
be the associated symmetric \sC net bundle. Then the following are equivalent:
\textbf{(1)} There is $\efG \in {\bf bun}(K,G)$ and a $\efG$-Hilbert net bundle $\efH$ with an isomorphism 
\[
( \efA_\chi , \tau , \vareps ) \ \simeq \  ( \mO_\efH^\efG , \varrho , \vartheta ) \ .
\]
\textbf{(2)} There is a lift
\begin{equation}
\label{eq_FL08}
\wt \chi : \pi_1(K) \to NG
\ \ : \ \
q \circ \wt \chi = \chi
\ .
\end{equation}
\textbf{(3)} There is a monomorphism 
\[
(\wa \tau,\wa \jmath,\otimes,\vareps)_K \hra (\wa \varsigma,\wa \jmath,\otimes,\vartheta)_K 
\]
of symmetric tensor net bundles, 
where any $\wa \varsigma_a$, $a \in K$, is isomorphic to the full subcategory of ${\bf Hilb}$
of tensor powers of $\bC^d$.
\end{thm}

\begin{proof}
{\bf (1) $\Rightarrow$ (2)}: By Lemma \ref{lem_FL07} the Hilbert net bundle $\efH$ has a reduction ${\wt \chi}' \in \hom(\pi_1(K),NG)$; moreover, Prop.\ref{prop_FL07} implies that $\mO_\efH^\efG$ has holonomy $\chi' := q \circ {\wt \chi}'$. Using (\ref{eq_thm_FL08}) and the isomorphism 
$\efA_{\chi ; \tau,\vareps} \simeq \mO^\efG_{\efH ; \varrho,\vartheta}$ 
we find $y \in QG$ such that
\[
\chi'(p) = y \chi(p) y^{-1}
\ \ , \ \
\forall p \in \pi_1(K)
\ .
\]
Thus picking some $u \in q^{-1}(y)$ and defining 
$\wt \chi (p) := u^{-1} \wt {\chi'}(p) u$, $p \in \pi_1(K)$, yields the desired lift.
{\bf (2) $\Rightarrow$ (1)}: We regard $NG$ as a subgroup of $\ud$ and consider the Hilbert net bundle $\efH \in {\bf bun}(K,\bC^d)$ with holonomy representation $\wt \chi$. By Lemma \ref{lem_FL07} there is a group net subbundle $\efG$ of ${\bf U}\efH$, and Prop.\ref{prop_FL07} implies that 
$(\mO_\efH^\efG,\varrho,\vartheta)$ 
has holonomy representation $\chi$. Thus $\efA_\chi$ and $\mO_\efH^\efG$ have the same holonomy 
$\chi \in \hom(\pi_1(K),QG)$, 
and by (\ref{eq_thm_FL08}) we obtain the desired isomorphism.
{\bf (3) $\Rightarrow$ (2)}: Since $\wa \tau$ is a symmetric tensor net subbundle of $\wa \varsigma$,
the holonomy 
\[
\chi : \pi_1(K) \to 
{\bf aut} \wa \varsigma_{a ; \otimes,\theta} \ \simeq \
{\bf aut}_{\sigma,\theta}\mO_d \ \simeq \
\ud
\]
restricts to a holonomy with values in the automorphism group of the subcategory 
$\wa \tau_{a ; \otimes,\theta} \subseteq \wa \varsigma_{a ; \otimes,\theta}$.
By Prop.\ref{prop_LTGA1a}, these automorphisms yields automorphisms in ${\bf aut}_{\rho,\theta}\mO_G$,
and this is equivalent to say, by Lemma \ref{lem_NG}, that
$\chi(p) \in {\bf aut}^G\mO_d$ for all $p \in \pi_1(K)$.
But, again by Lemma \ref{lem_NG}, this means that $\efH$ has holonomy $u : \pi_1(K) \to NG$
such that $\chi(p) = \wa u(p)$, $\forall p \in \pi_1(K)$.
{\bf (1) $\Rightarrow$ (3)}: The isomorphism
$\efA_{\chi ; \tau,\vareps} \simeq \mO^\efG_{\efH ; \varrho,\vartheta}$ 
yields an isomorphism
$(\wa \tau,\wa \jmath,\otimes,\vareps)_K \simeq (\wa \varrho,\wa \jmath,\otimes,\vartheta)_K$,
and this last is a symmetric tensor net subbundle of
$(\wa \varsigma,\wa \jmath,\otimes,\vartheta)_K$.
The fact that $\wa \varsigma$ has the desired fibre follows by Prop.\ref{prop_FL01}.
\end{proof}

%
%

With the notation of the previous theorem, we say $\efG$ is a \emph{gauge group} of 
$\efA_{\chi ; \varrho,\vareps}$.
In the following lines we give examples in which $\efG$ is not uniquely determined by 
$\efA_{\chi ; \varrho,\vareps}$,
and in which $\efG$ does not exist at all.

\

\noindent {\bf The case $G = \sud$ and non-uniqueness of the gauge group.} 
Let $d \in \bN$ and $\efH \in {\bf bun}(K,\bC^d)$ be a rank $d$ Hilbert net bundle. By Theorem \ref{thm_NET00} there is a holonomy representation $\chi : \pi_1(K) \to \ud$ identifying $\efH$ up to isomorphism. Composing $\chi$ with the determinant yields a morphism 
\begin{equation}
\label{eq_FL03A}
\efc_1(\efH) \in \hom (\pi_1(K),\bT)
\ ,
\end{equation}
called the \emph{first Chern class of} $\efH$. In particular, if $M$ is a manifold and $\Delta$ is a good base of $M$, then (\ref{eq_BK01}) and the Hurewicz theorem imply that $\hom (\pi_1(\Delta),\bT)$ is isomorphic to the singular cohomology group $H^1(M,\bT)$, and we may regard $\efc_1(\efH)$ as an element of $H^1(M,\bT)$. This yields the characteristic class
\[
\efc_1 : {\bf bun}(\Delta,{\bf Hilb}) \to H^1(M,\bT)
\ \ : \ \
\efc_1 (\efH \oplus \efH') = \efc_1 (\efH) \cdot \efc_1(\efH')
\ ,
\]
which, in essence, is the first Cheeger-Chern-Simons class (see \cite[Example 1.5]{CS85}) of the flat Hermitian vector bundle $\mE \to M$ associated with $\efH$ in the sense of \S \ref{sec_lc}.

Let $G = \sud$. Then $NG = \ud$, $QG = \bT$ and $q : \ud \to \bT$ is the determinant. Thus we have the one-to-one correspondence
\[
\hom(\pi_1(K),\bT) \to {\bf bun}(K;\mO_\sud,\rho,\theta)
\ \ , \ \
\chi \mapsto (\efA_\chi,\tau,\vareps)
\ .
\]
Now, the determinant map $q : \ud \to \bT$ admits left inverses $s : \bT \to \ud$, $q \circ s = id_\bT$ (for example, we may define $s(z)$, $z \in \bT$, as the diagonal matrix with entries $(z,1,\ldots,1)$). Thus every $\chi \in \hom(\pi_1(K),\bT)$ has a lift $\wt \chi := s \circ \chi$, and the map
\[
q_* : \hom(\pi_1(K),\ud)  \to  \hom (\pi_1(K),\bT)
\]
is surjective. Applying Theorem \ref{thm_NET00} to the category of Hilbert spaces we find that every morphism $\nu \in \hom ( \pi_1(K), \ud )$ defines the Hilbert net bundle $\efH_\nu \in {\bf bun}(K,\bC^d)$. Since $q_*$ is induced by the determinant, by definition of first Chern class  (see (\ref{eq_FL03A})) we find $\efc_1(\efH_\nu) = q_* (\nu)$. Thus we proved:

\begin{prop}
\label{prop_sud}
Let $\chi \in \hom(\pi_1(K),\bT)$ and 
$\efA_{\chi ; \tau,\vareps} \in {\bf bun}(K;\mO_\sud,\rho,\theta)$
denote the associated symmetric \sC net bundle. 
Then for any $\efH \in {\bf bun}(K,\bC^d)$ we have
\[
\efc_1(\efH) = \chi
\ \ \Leftrightarrow \ \
(\mO_\efH^{{\bf SU} \efH},\varrho,\vartheta)
\ \simeq \
(\efA_\chi,\tau,\vareps)
\ .
\]
\end{prop}

\begin{rem}{\it 
\label{rem_break}
Examples of Hilbert net bundles $\efH , \efH'$ with non-isomorphic special unitary bundles but such that $\efc_1(\efH) = \efc_1(\efH')$ can be easily constructed (see below). By Prop.\ref{prop_FL01} we conclude that despite ${\bf SU} \efH$ and ${\bf SU} \efH'$ are not isomorphic, there can be an isomorphism 
\[
(\mO_\efH^{{\bf SU} \efH},\varrho,\vartheta)
\ \simeq \
(\mO_{\efH'}^{{\bf SU} \efH'},\varrho',\vartheta')
\ .
\]
} 
\end{rem}

Let us consider the case $\pi_1(K) = \bZ$. This occurs in interesting examples in which $K$ is a good base for the topology of 
the $1$--sphere $S^1$, 
anti-de Sitter spacetimes $adS_n \simeq S^1 \times \bR^{n-1}$, $n > 1$, and
de Sitter spacetime $dS_2 \simeq S^1 \times \bR$.
Now, any $\chi \in \hom(\bZ,\bT)$ is uniquely determined by $\chi(1) \in \bT$ and we find
\[
{\bf bun}(K;\mO_\sud,\rho,\theta)
\ \simeq \ 
\hom ( \bZ , \bT ) 
\ \simeq \ 
\bT 
\ .
\]
On the other hand, any $\chi \in \hom(\bZ,\ud)$ has a decomposition $\chi = \oplus_k^d \chi_k$, where $\chi_k \in \hom ( \bZ , \bT )$. Clearly, changing the order of the $\chi_k$ does not change the class of $\chi$ in $\unl \hom(\bZ,\ud)$; thus, applying (\ref{eq_thm_FL08}) with $G = \{ 1 \}$, we conclude that
\[
\unl{\bf bun}(K;\mO_d,\sigma,\theta) 
\ \simeq \ 
\unl{\hom}( \bZ , \ud )
\ \simeq \ 
\bT^d / \bP_d \ ,
\]
where $\bT^d / \bP_d$ is the quotient of the torus $\bT^d$ under the action of the permutation group $\bP_d$. Hence the first Chern class is given by the (not injective!) epimorphism
\[
\efc_1 : \bT^d / \bP_d \to \bT
\ \ , \ \
\efc_1 [ z_1 , \ldots , z_d  ] \, := \, \prod_k z_k \ .
\]

\noindent {\bf Absence of the lift.} Let $G \simeq \bT$ be the torus acting on the Hilbert space $H$ of dimension $d > 1$, so that $NG = \ud$ and $QG = \bPU(d)$, the projective unitary group. We consider a finite group $\Gamma$ with nontrivial cohomology $H^2(\Gamma,\bT)$ (for example, the permutation group $\bP_n$, $n \geq 4$, see \cite[Theorem 2.9]{HH}) and a projective representation 
$\chi : \Gamma \to \bPU(d)$
with nontrivial class $\delta_\chi \in H^2(\Gamma,\bT)$, so that we cannot find a lift $\wt \chi : \Gamma \to \ud$.
By the Eilenberg-McLane construction (\cite[\S 1.B]{H}), there are a space $M$ with $\pi_1(M) = \Gamma$ and a good base $\Delta$ of $M$ (\cite[Prop.A.4]{H}). 
Using the isomorphism 
$\Gamma = \pi_1(\Delta) \simeq \pi_1(M)$,
we obtain the desired class of examples by means of the one-to-one correspondence
\[
\hom(\Gamma,\bPU(d)) \to {\bf bun}(\Delta;\mO_G,\rho,\theta)
\ \ , \ \
\chi \mapsto (\efA_\chi,\tau,\vareps)
\ .
\]


%

\section{Symmetric tensor presheaves: invariants and Tannaka duality.}
\label{sec_CCS}

In this section we study symmetric tensor presheaves $\efT$ with simple units.
As a first step we show that any section $\varrho \in \obj \wt \efT$ defines a holonomy representation
$\chi_\varrho : \pi_1(K) \to QG_\varrho := NG_\varrho / G_\varrho$
describing the category of tensor powers of $\varrho$,
where $G_\varrho \subseteq \ud$ is a suitable compact Lie group. 
When $K$ is a base, applying the equivalence with the category of flat bundles (\S \ref{sec_lc})
we assign to $\varrho$ Cheeger-Chern-Simons classes describing
the obstruction to $\chi_\varrho$ being trivial in geometric terms.
The second step is our main result, in which we prove that any embedding of $\efT$
defines a "gauge" group net bundle, the analogue of the dual group of Tannaka duality.
Existence and uniqueness of the embedding are not ensured
due to the lifting problem for the holonomies 
$\chi_\varrho$, $\varrho \in \obj \wt \efT$,
and we describe this situation in terms of a non-abelian 1--cocycle with coefficients
in the crossed module $G_\varrho \to NG_\varrho$.
Clearly, passing to the opposite poset our results apply to symmetric tensor nets with simple units.

\

\noindent {\bf Presheaves of tensor \sC categories and their invariants.}
The following notion is a categorical counterpart of the one of
conjugate representation of a compact group (\cite[\S 1]{DR89}), and is motivated
by the correspondence particle-antiparticle (\cite[Eq.3.21-3.22]{DR90}):
we say that a tensor \sC category $T_\otimes$ has \emph{conjugates} 
whenever for any $\rho \in \obj T$ there is $\ovl \rho \in \obj T$ and 
$R      \in (\iota,\ovl \rho \rho)$, 
$\ovl R \in (\iota,\rho \ovl \rho)$ 
solving the \emph{conjugate equations}
\begin{equation}
\label{eq.conj}
(\ovl R^* \otimes 1_\rho) \circ (1_\rho \otimes R) = 1_\rho
\ \ , \ \
(R^* \otimes 1_{\ovl \rho}) \circ (1_{\ovl \rho} \otimes \ovl R) = 1_{\ovl \rho} \ ,
\end{equation}
where $\iota$ is the identity object of $T$.

A symmetric tensor \sC category $D_{\otimes,\eps}$ with simple unit having conjugates,
direct sums and subobjects is called a \emph{DR-category}. 
By \cite[Theorem 6.9]{DR89} there is a symmetric tensor embedding
$J : D_{\otimes,\eps} \to {\bf Hilb}$,
unique up to unitary tensor natural transformations, such that 
$J(D_{\otimes,\eps})$
is the dual of a compact group $G$ 
(that is, the subcategory of {\bf Hilb} with arrows $G$-equivariant linear operators).

\begin{defn}
Let $\efT = (T,r,\otimes,\vareps)^K$ be a symmetric tensor presheaf with simple units.
We say that $\efT$ is a \textbf{DR-presheaf} whenever the section category $\wt \efT$ is a DR-category.
\end{defn}

By Prop.\ref{prop.NET1}, if $\efT$ is a DR-presheaf then the associated 
presheaf bundle ${}_\beta \efT$ is a DR-presheaf.

\begin{lem}
\label{lem_G}
Let $\efT$ be a DR-presheaf and ${}_\beta F := {}_\beta T_a$ denote the standard fibre of ${}_\beta \efT$.
Then there are, unique up to isomorphisms, a compact group $G$ and a symmetric tensor embedding
\[
I \, : \, {}_\beta F \to {\bf Hilb} \ ,
\]
such that $I({}_\beta F)$ is a full subcategory of the dual of $G$.
\end{lem}

\begin{proof}
We already know that ${}_\beta F$ is a symmetric tensor \sC category with simple unit,
so, as a first step, we prove that ${}_\beta F$ has conjugates.
To this end we note that, given $\varrho^a \in \obj {}_\beta F$ (that is, $\varrho \in \obj \wt \efT$), 
by hypothesis there is a triple
$\ovl \varrho \in \obj \wt \efT$,
$R      \in (i,\ovl \varrho \varrho)$, 
$\ovl R \in (i,\varrho \ovl \varrho)$ 
solving (\ref{eq.conj}) in $\wt \efT$, so the triple
$\ovl \varrho^a$, $R_a$, $\ovl R_a$
solves (\ref{eq.conj}) in ${}_\beta F$. 
This proves that ${}_\beta F$ has conjugates.
Moreover ${}_\beta F$ has direct sums, in fact given 
$\varrho^a , \varsigma^a \in \obj {}_\beta F$
there is a direct sum $\tau \in \obj \wt \efT$ with orthogonal isometries 
$V_\varrho \in (\varrho,\tau)$, $V_\varsigma \in (\varsigma,\tau)$,
so considering the evaluations 
$V_{\varrho,a}$, $V_{\varsigma,a}$ 
we have that $\tau^a$ is a direct sum of $\varrho^a,\varsigma^a \in \obj {}_\beta F$.
In general ${}_\beta F$ may be not closed under subobjects: so we pass to its closure
${}_\beta \ovl F$, having objects projections of 
${}_\beta(\varrho^a,\varsigma^a)$, $\forall \varrho^a,\varsigma^a \in \obj {}_\beta F$.
Routine computations (\cite[Theorem 2.4]{LR97}) show that ${}_\beta \ovl F$ has conjugates (besides direct sums and subobjects),
so it is a DR-category and we get the desired compact group $G$ and embedding $I$.
Their uniqueness follows by the fact that ${}_\beta F$ is unique up to isomorphism.
\end{proof}

We now consider the subpresheaf bundle ${}_\beta \wa \varrho$ of ${}_\beta \efT$ with fibres the full subcategories 
${}_\beta \wa \varrho^o \subseteq {}_\beta T_o$, $o \in K$,
of tensor powers of $\varrho^o \in \obj {}_\beta T_o$. Before to give the following result,
we recall the reader to the notation of Theorem \ref{thm_OG}.
\begin{thm}[Holonomy]
\label{thm.hol}
Let $\efT = (T,r,\otimes,\vareps)^K$ be a DR-presheaf. 
Then for any section $\varrho \in \obj \wt \efT$ there are a compact Lie group
$G_\varrho \subseteq \ud$
such that 
${}_\beta \wa \varrho \in {\bf pbun}^\uparrow(K,\wa \pi_{G_\varrho;\otimes,\theta})$,
and a morphism
$\chi_\varrho : \pi_1(K) \to QG_\varrho := NG_\varrho / G_\varrho$
such that ${}_\beta \wa \varrho$ has holonomy $\chi_\varrho$.
\end{thm}

\begin{proof}
We maintain the notation of the previous Lemma and that of Prop.\ref{prop.NET1}.
Let $\varrho \in \obj \wt \efT$. By the previous Lemma there are a $G$-Hilbert space $H_\varrho$ 
and a compact Lie group $G_\varrho$, the image of $G$ under the $G$-action on $H_\varrho$, such that 
$f \in {}_\beta(\varrho^{a,r},\varrho^{a,s})$, $r,s \in \bN$, 
if and only if 
$I(f)$ is a $G_\varrho$-invariant operator from the tensor power $H_\varrho^r$ to $H_\varrho^s$,
that is, we have an isomorphism 
${}_\beta \wa \varrho^a_{\otimes,\vareps} \simeq \wa \pi_{G_\varrho;\otimes,\theta}$.
If $\gamma \in K(a)$ then we define
$Z_\gamma \in {\bf aut}{}_\beta F_{\otimes,\eps}$,
\begin{equation}
\label{def_zgamma}
Z_\gamma
\ := \
{}_\beta r_{a,|b_n|} \circ {}_\beta r_{|b_n|,\partial_1 b_n} 
\circ \ldots \circ 
{}_\beta r_{\partial_0b_1,|b_1|} \circ {}_\beta r_{|b_1|,a}
\ , \
\end{equation}
where $\gamma := (b_n , \ldots , b_1)$. For any $\varsigma^a \in \obj {}_\beta F$ we have
\begin{equation}
\label{eq_CCS.rho}
Z_\gamma(\varrho^a) = \varrho^a
\ \ , \ \
Z_\gamma(f) \in {}_\beta(\varrho^a,\varsigma^a) \subseteq (\varrho_a,\varsigma_a)
\ \ , \ \
\forall f \in {}_\beta(\varrho^a,\varsigma^a)
\ .
\end{equation}
By \cite[Theorem 6.9]{DR89} there is a unitary, tensor natural transformation $u(\gamma) \in (I,I \circ Z_\gamma)$, 
defining unitaries $u_\varrho(\gamma) \in UH_\varrho$, $\forall \varrho^a \in \obj {}_\beta F$, such that 
\begin{equation}
\label{eq_CCS01p}
\{ I \circ Z_\gamma \}(f) = u_\sigma(\gamma) \circ I(f) \circ u_\varrho(\gamma)
\ \ , \ \
\forall f \in {}_\beta(\varrho^a,\varsigma^a)
\ .
\end{equation}
In particular,
\begin{equation}
\label{eq_CCS01}
\{ I \circ Z_\gamma \}(f) = \wa u_\varrho(\gamma) \circ I(f)
\ \ , \ \
f \in {}_\beta(\varrho^{a,r},\varrho^{a,s})
\  ,\
r,s \in \bN
\ ,
\end{equation}
where $\wa u_\varrho (\gamma)$ is the automorphism defined by (\ref{eq.NGa}).
Now, the l.h.s. of (\ref{eq_CCS01}) is a symmetric tensor isomorphism from
${}_\beta \wa \varrho^a$ to $\wa \pi_{G_\varrho}$,
so $\wa u_\varrho(\gamma)$ restricts to a symmetric tensor automorphism of $\wa \pi_{G_\varrho}$
and by (\ref{eq.NG}) we have $u_\varrho(\gamma) \in NG$, $\forall \gamma \in K(a)$.
On the other side, the l.h.s. of (\ref{eq_CCS01})
preserves the composition of paths and does not depend on the choice of $\gamma \in K(a)$ 
in its homotopy class, thus the map 
\begin{equation}
\label{eq_CCS01a0}
\pi_1(K) \to {\bf aut}\wa \pi_{G_\varrho;\otimes,\theta}
\ \ , \ \
[\gamma] \mapsto \wa u_\varrho(\gamma)|_{ \wa \pi_{G_\varrho} }
\ ,
\end{equation}
is well-defined and yields (up to the isomorphism with $\wa \pi_{G_\varrho}$) a holonomy for ${}_\beta \wa \varrho$. 
Applying (\ref{eq.NG}) we obtain the desired holonomy
\begin{equation}
\label{eq_CCS01a}
\chi_\varrho \, : \, \pi_1(K) \to QG_\varrho 
\ \ : \ \
\chi_\varrho(\gamma) := q(u_\varrho(\gamma))
\ , \ 
\forall \gamma \in K(a)
\ .
\end{equation}
Clearly changing the standard fibre yields a group $G'_\varrho$ isomorphic to $G_\varrho$ 
and an equivalent holonomy.
\end{proof}

\begin{rem}{\it
\label{rem_CCS00}
Let $\varrho \in \obj \wt \efT$ denote a section of the DR-presheaf $\efT$.
For any $o \in K$ we consider the Doplicher-Roberts \sC algebra $A_{\varrho,o}$ generated by the arrows of 
${}_\beta \wa \varrho^o$,
see \cite[\S 4]{DR89}. By Lemma \ref{lem_G} we have that the operation of tensoring on the right by $1_{\varrho^o}$
is faithful on arrows of ${}_\beta \wa \varrho^o$, in fact it is faithful on arrows of ${\bf Hilb}$.
So we apply \cite[Theorem 4.3, Theorem 4.4]{DR89} and conclude that there are
\[
\varrho_{o,*} \in {\bf end}A_{\varrho,o} 
\ \ \ , \ \ \
\eps_{o,*} \in UA_{\varrho,o} \cap ( \varrho_{o,*}^2 \, , \, \varrho_{o,*}^2 ) \ ,
\]
with a canonical tensor *-functor
${}_\beta \wa \varrho^o \to \wa \varrho_{o,*}$.
By functoriality of the Doplicher-Roberts algebra, see \cite[Theorem 5.1]{DR89}, any symmetric tensor isomorphism
${}_\beta r_{o'o} : {}_\beta \wa \varrho^o \to {}_\beta \wa \varrho^{o'}$,
$o \leq o'$,
induces an isomorphism of pointed \sC dynamical systems
\[
\jmath_{o'o} \, : \,
( A_{\varrho,o}  \, , \, \varrho_{o,*}  \, , \, \eps_{o,*}  ) \to 
( A_{\varrho,o'} \, , \, \varrho_{o',*} \, , \, \eps_{o',*} )
\ ,
\]
defining the pointed dynamical \sC net bundle 
$\efA_\varrho = (A_\varrho,\jmath)_K$, $( \efA_\varrho \, , \, \varrho_* \, , \, \eps_* )$.
The unitary $\eps_{o,*}$ makes $\varrho_{o,*}$ a \textbf{symmetric} endomorphism whenever $\varrho^o$ is \textbf{special}
in the sense of \cite[\S 3]{DR89}, see \cite[Corollary 4.7]{DR89}. 
Special objects play a crucial role for the duality theorem \cite[Theorem 6.1]{DR89}, 
and generate the underlying symmetric tensor category in a suitable sense, see \cite[Theorem 3.4]{DR89};
if $\varrho^o$ is special, then it is trivially verified that any $\varrho^e$, $e \in K$, is special,
and in this case we say that $\varrho$ is special.
Fixing $a \in K$ and applying \cite[Theorem 4.17]{DR89}, we have isomorphisms
\[
( {}_\beta \wa \varrho^a \, , \, \otimes_a \, , \, \vareps_a  ) \ \simeq \ 
( \wa \varrho_{a,*}      \, , \, \otimes_a \, , \, \eps_{a,*} ) \ \simeq \ 
( \wa \pi_{G_\varrho}    \, , \, \otimes   \, , \, \varphi    )
\ \ \ , \ \ \
( A_{\varrho,a} \, , \, \varrho_{a,*} \, , \, \eps_{a,*} ) \ \simeq \ (\mO_{G_\varrho} \, , \, \rho \, , \, \theta)
\ .
\]
Since the *-morphisms $\jmath_{o'o}$, $o \leq o'$, are induced by ${}_\beta r_{o'o}$, see \cite[Theorem 5.1]{DR89},
we have the isomorphism of symmetric tensor presheaves
${}_\beta \wa \varrho_{\otimes,\vareps} \ \simeq \ \wa \varrho_{*;\otimes,\eps_*}$.
Thus 
$( \efA_\varrho \, , \, \varrho_* \, , \, \eps_* ) \in {\bf bun}(K;\mO_{G_\varrho},\rho,\theta)$
and it has the holonomy $\chi_\varrho$ of (\ref{eq_CCS01a}), by applying Corollary \ref{cor_NET03}.
Applying Theorem \ref{thm_FL08}, we have a characterization of those special sections $\varrho \in \obj \wt \efT$
such that ${}_\beta \wa \varrho$ can be embedded in a presheaf bundle with fibre a \textbf{full} subcategory of ${\bf Hilb}$:
this happens if, and only if, $\chi_\varrho$ has a lift to $NG_\varrho$.
We shall prove this result for generic $\varrho$, in Theorem \ref{cor.lift}.
}
\end{rem}

\

\noindent {\bf Cohomological aspects of holonomy.}
We say that $\varrho \in \obj \wt \efT$ is \emph{liftable} whenever $\chi_\varrho$ has a lift to $NG_\varrho$.
%
%
%
%
In the following result we show that lifts of $\chi_\rho$ can be characterized in terms of a non-abelian 1--cocycle with values in the crossed module $G_\varrho \to NG_\varrho$ (see \cite[\S 4]{Vas09}).
\begin{thm}[Gerbes]
\label{thm.gerbe}
Any section $\varrho$ of a DR-presheaf defines the cochain
\[
{\bf u}_\varrho : \Sigma_1(K) \to NG_\varrho
\ \ : \ \
d {\bf u}_\varrho(c) := 
{\bf u}_\varrho(\partial_0c) {\bf u}_\varrho(\partial_2c) {\bf u}_\varrho(\partial_1c)^* 
\in G_\varrho
\ , \ 
\forall c \in \Sigma_2(K)
\ .
\]
(Eventual) lifts of $\chi_\rho$ are in one-to-one correspondence with pairs $(v,g)$,
$v : \Sigma_0(K) \to NG_\varrho$,
$g : \Sigma_1(K) \to  G_\varrho$,
such that 
\begin{equation}
\label{eq.gerbe}
z(b) \ := \ v(\partial_0b) \, g(b) \, {\bf u}_\varrho(b) \, v(\partial_1b)^* \ \in NG_\varrho
\ \ , \ \qquad
\forall b \in \Sigma_1(K)
\ ,
\end{equation}
is a cocycle, i.e. $dz(c) = 1$, $\forall c \in \Sigma_2(K)$.
\end{thm}

\begin{proof}
For any $b \in \Sigma_1(K)$ we define $\ovl b$ as the 1-simplex 
$\partial_0 \ovl b = \partial_1b$,
$\partial_1 \ovl b = \partial_0b$,
$|\ovl b| = |b|$.
We note that any $c \in \Sigma_2(K)$ defines the path 
$\gamma_c := ( \partial_0c , \partial_2c , \overline{\partial_1c} )$,
$\gamma_c : \partial_{00}c \to \partial_{00}c$,
and then take a path frame $\gamma_a = \{ \gamma_{ao} : o \to a \}$.
Then we define
\[
{\bf u}_\varrho(b) \ := \ 
u_\varrho( \gamma_{a \partial_0b} * b * \gamma_{a \partial_1b}^{-1})
\ \ , \ \
\forall b \in \Sigma_1(K)
\ ,
\]
where $u_\varrho : K(a) \to NG_\varrho$ is the map defined by (\ref{eq_CCS01p}).
We have
\[
d{\bf u}_\varrho(c) = u_\varrho(\wt \gamma_c)
\ \ , \ \
\wt \gamma_c := \gamma_{a \partial_{00}c} * \gamma_c * \gamma_{a \partial_{00}c}^{-1}
\ ;
\]
since $\gamma_c$ is homotopic to the trivial path 
$\iota_{\partial_{00}c} \in K(\partial_{00}c)$, 
we conclude that $\wt \gamma_c$ is homotopic to the trivial path $\iota_a$. 
Since 
$\wa u_\varrho |_{\wa \pi_{G_\varrho}} : K(a) \to {\bf aut}\wa \pi_{G_\varrho}$ 
factorizes through $\pi_1(K)$, this implies that 
$\wa u_\varrho( \gamma_{a \partial_{00}c} * \gamma_c * \gamma_{a \partial_{00}c}^{-1})$
restricts to the identity of $\wa \pi_{G_\varrho}$, i.e., by (\ref{eq.NG}),
$d{\bf u}_\varrho(c) = u_\varrho(\wt \gamma_c) \in G_\varrho$ 
as desired.
Let now $(v,g)$ be a pair such that the cochain $z$ of (\ref{eq.gerbe}) is a cocycle.
Then we have 
\[
q \circ {\bf u}_\varrho(\gamma) 
\ = \ 
q(v(a)) \cdot q \circ z(\gamma) \cdot q(v(a)) \in QG_\varrho
\ \ , \ \
\forall \gamma : a \to a \ ,
\]
so $\chi_\varrho$ has lift $\wt \chi_\varrho (\gamma) := v(a)^* z(\gamma) v(a)$
as desired.
On the converse, if $\chi_\varrho$ has lift $\wt \chi_\varrho$ then by the equivalence 
of Theorem \ref{thm_NET00} and (\ref{eq_FC03}) there is a cocycle
$z(b) \in NG_\varrho$, $b \in \Sigma_1(K)$,
such that there is $w : \Sigma_0(K) \to QG_\varrho$ with
$q \circ {\bf u}_\varrho(b) = w(\partial_0b) \cdot q \circ z(b) \cdot w(\partial_1b)^*$.
For any $a \in K$ we pick $v(a) \in NG_\varrho$ such that $q(v(a)^*) = w(a)$, so
\[
q \circ {\bf u}_\varrho(b) \ = \ q ( \ v(\partial_0b)^* \cdot z(b) \cdot v(\partial_1b) \ )
\ \ , \ \
\forall b \in \Sigma_1(K)
\ ,
\]
This proves that there is $g : \Sigma_1(K) \to G_\varrho$ such that
(\ref{eq.gerbe}) holds, as desired.
\end{proof}

$NG_\varrho$-cocycles $z,z'$ fulfilling (\ref{eq.gerbe}) are not necessarily
cohomologous, in fact they may define inequivalent $NG_\varrho$-holonomies
as in the case $G_\varrho = \sud$ (see Prop.\ref{prop_sud} and following remarks).
It is easy to verify that a change $a \to a'$ of the standard fibre (or of the path frame $\gamma_a$)
yields a cochain ${\bf u}'_\varrho : \Sigma_1(K) \to NG'_\varrho$ with an isomorphism
$\beta : NG_\varrho \to NG'_\varrho$ 
such that ${\bf u}'_\varrho = \beta \circ {\bf u}_\varrho$.

\

\noindent {\bf Cheeger-Chern-Simons classes.}
By (\ref{eq_PR}), the map (\ref{eq_CCS01a}) yields a principal net $QG_\varrho$-bundle that we denote by $\efP_\varrho$.

Let now $M$ be a manifold and $\Delta$ denote a good base of $M$; we consider a DR-presheaf $\efT$ over $\Delta$ and $\varrho \in \obj \wt \efT$.
In this setting, $\efP_\varrho$ defines the flat principal $QG_\varrho$-bundle 
$\mP_\varrho \to M$,
see (\ref{eq02_thm_NET01}), and characteristic classes arise, in the way that we explain in the following lines.

Let $\phi \in \bR[x_1 , \ldots , x_d]$ be a $QG_\varrho$-invariant polynomial of degree $k \in \bN$ 
(here $d$ is the rank of the Lie algebra of $QG_\varrho$). 
Then $\phi$ defines a characteristic class assigning to \emph{any} principal $QG_\varrho$-bundle a closed $2k$-form,
\[
c_\phi : {\bf Pr}(M,QG_\varrho) \to Z_{deRham}^{2k}(M)
\]
(see \cite[Vol.2,\S XII.1]{KN}).
We say that $c_\phi$ has periods in $\bK = \bZ,\bQ$ whenever for any $\mP \in {\bf Pr}(M,QG_\varrho)$ it turns out
\[
\int_\zeta c_\phi(\mP) \in \bK \ \ , \ \ \forall \zeta \in Z_{2k}(M) \ ,
\]
where $Z_{2k}(M)$ is the set of singular $2k$-cycles.
We denote the set of $QG_\varrho$-invariant polynomials of degree $k$ 
defining a class with periods in $\bK$ by $I^k(QG_\varrho;\bK)$.
Now, Cheeger and Simons proved that for any $\phi \in I^k(QG_\varrho;\bK)$ there is a singular $(2k-1)$-cochain 
$c^\uparrow_\phi(\mP)$
with coefficients in ${\mathbb{R/K}}$, such that 
\[
\left \langle c^\uparrow_\phi(\mP) , \partial \ell \right \rangle  
\ = \ 
\int_\ell c_\phi(\mP) \ {\mathrm{mod}} \bK 
\ \ , \ \
\forall \ell \in C_{2k}(M)
\ ,
\]
where $C_{2k}(M)$ is the set of singular $2k$-chains and
$\partial : C_{2k}(M) \to Z_{2k-1}(M)$
is the boundary.
When $\mP$ is flat $c_\phi(\mP) = 0$ (\cite[Vol.2, \S XII.1]{KN}),
and the previous equality says that $c^\uparrow_\phi(\mP)$ vanishes on $\partial C_{2k}(M)$. 
So $c^\uparrow_\phi(\mP)$ is a cocycle, whose cohomology class
\[
[c^\uparrow_\phi(\mP)] \in H^{2k-1}(M,{\mathbb{R/K}})
\]
is called the \emph{Cheeger-Chern-Simons} secondary class defined by $\phi$ (see \cite[Theorem 1.1]{CS85}).
%
%
%
%
%
Returning to $\varrho \in \obj \wt \efT$, for any $\phi \in I^k(QG_\varrho;\bK)$ we define
\begin{equation}
\label{eq.ccs}
c_\phi(\varrho) \ := \ [c^\uparrow_\phi(\mP_\varrho)] \ \in H^{2k-1}(M,{\mathbb{R/K}}) \ .
\end{equation}
By construction $c_\phi$ vanishes when $\chi_\varrho$ is trivial, so we can measure the holonomy of $\varrho$ in geometric terms.

\

\noindent {\bf Duality theory.} We now prove a Tannaka duality for presheaves
of symmetric tensor \sC categories. Before to proceed we give the following notion:
an \emph{embedding} is a morphism of symmetric tensor presheaves with simple units
\[
I : \efT_{\otimes,\vareps} \to \efC_{\otimes,\vartheta} \ ,
\]
where $\efC = (C,R,\otimes,\vartheta)_K$ is a Hilbert presheaf and $I_o$ is an embedding for any $o \in K$.
Composing the induced functor 
$\wt I : \wt \efT_{\otimes,\vareps} \to \wt \efC_{\otimes,\vartheta}$
with (\ref{lem_LTGA2}) we have that any $\varrho \in \obj \wt \efT$ defines the Hilbert net bundle 
$\efH_\varrho = (H_\varrho,R^\varrho)_K$,
where
$H_{\varrho,a} := I_a(\varrho_a)$, $a \in K$,
and 
$R^\varrho_{a'a} : H_{\varrho,a} \to H_{\varrho,a'}$, $a \leq a'$,
is defined as in (\ref{eq_LTGA2}).
Moreover, any $t \in (\varrho,\varsigma)$ defines the morphism $\wt I(t) \in (\efH_\varrho,\efH_\varsigma)$.
Let now $\varrho,\varsigma \in \obj \wt \efT$ and 
$f' \in {}_\beta(\varrho^{a'},\varsigma^{a'}) \subseteq ( \varrho_{a'} , \varsigma_{a'} )$.
Defining $f := {}_\beta r_{aa'}(f')$ we find
$I_a(f) = R_{aa'} \circ I_{a'}(f')$,
so applying (\ref{eq.LTGA2}) we have
\begin{equation}
\label{eq_CCSA}
R^\varsigma_{a'a} \circ I_a(f) \circ R^\varrho_{aa'} \ = \ I_{a'}(f') \ .
\end{equation}
We now consider the groups of unitary tensor natural transformations $G_a$, $a \in K$, 
having elements families
$u^a \ :=  \{ u_\varrho^a \in UH_{\varrho,a} \}_{\varrho \in \obj \wt \efT}$
such that
\begin{equation}
\label{eq_CCSB}
I_a(f) \circ u^a_\varrho = u^a_\varsigma \circ I_a(f)
\ \ , \ \
u_\varrho^a \otimes_a u_\varsigma^a = u_{\varrho \varsigma}^a
\ \ , \ \
\forall f \in {}_\beta(\varrho^a,\varsigma^a)
\ .
\end{equation}
By construction any $G_a$, $a \in K$, is the Tannaka dual of $I_a({}_\beta T_a)$
(\cite[Remarks after Lemma 6.2]{DR89}), so it is a compact group.

\begin{thm}[Tannaka duality]
\label{thm.dual}
Let $\efT = (T,r,\otimes,\vareps)^K$ be a symmetric tensor presheaf with simple units and 
$I : \efT_{\otimes,\vareps} \to \efC_{\otimes,\vartheta}$
an embedding. 
Then there is a group net bundle $\efG = (G,i)_K$ and a full symmetric tensor embedding 
$\wt \efT \to {\bf bun}_\efG(K,{\bf Hilb})$.
\end{thm}

\begin{proof}
For any $a \leq a'$ we define the family of group isomorphisms
\[
\ad R^\varrho_{aa'} :  UH_{\varrho,a'} \to UH_{\varrho,a}
\ , \
u \mapsto R^\varrho_{aa'} \circ u \circ R^\varrho_{a'a}
\ \ , \ \
\varrho \in \obj \wt \efT \ .
\]
Let $\varrho,\varsigma \in \obj \wt \efT$. Then, with the notation of Prop.\ref{prop.NET1}, 
for any $a \leq a' \in K$ and 
$f \in {}_\beta(\varrho^a,\varsigma^a)$
there is a unique 
$f' \in {}_\beta(\varrho^{a'},\varsigma^{a'})$
such that 
$f = r_{aa'}(f')$, i.e. $f' = {}_\beta r_{a'a}(f)$, where ${}_\beta r_{a'a} := {}_\beta r_{aa'}^{-1}$. 
For any $u^a \in G_a$, we find                                      
\[
\begin{array}{ll}
I_{a'}(f') \circ \{ \ad R^\varrho_{a'a} (u^a_\varrho) \} & \ \ = \ \
I_{a'}({}_\beta r_{a'a}(f)) \circ \{ \ad R^\varrho_{a'a} (u^a_\varrho) \}  \\ & \stackrel{ (\ref{eq_CCSA}) }{=}
\{ R^\varsigma_{a'a} \circ I_a(f) \circ R^\varrho_{aa'} \}  \circ  \ad R^\varrho_{a'a} (u^a_\varrho) = \\ & \ \ = \ \ 
R^\varsigma_{a'a} \circ I_a(f) \circ u^a_\varrho \circ R^\varrho_{aa'}   =   \\ & \stackrel{ (\ref{eq_CCSB}) }{=} 
R^\varsigma_{a'a} \circ u^a_\varsigma \circ I_a(f) \circ R^\varrho_{aa'}  =  \\ & \ \ = \ \ 
\{ \ad R^\varsigma_{a'a}(u^a_\varsigma) \} \circ \{ R^\varsigma_{a'a} \circ I_a(f) \circ R^\varrho_{aa'} \} =  \\ & \ \ = \ \
\{ \ad R^\varsigma_{a'a}(u^a_\varsigma) \} \circ I_{a'}(f') 
\ ,
\end{array}
\]
and, applying (\ref{eq_LTGAa}),
\[
\{ \ad R^\varrho_{a'a} (u^a_\varrho) \} \otimes_{a'} \{ \ad R^\varsigma_{a'a} (u^a_\varsigma) \}
\ = \
\ad R^{\varrho \varsigma}_{a'a} (u^a_{\varrho \varsigma})
\ .
\]
The previous computations say that
$i_{a'a}u^a := \{ \ad R^\varrho_{a'a}(u^a_\varrho) \}_\varrho$
belongs to $G_{a'}$ so, since $i := \{ i_{a'a} \}$ fulfils the net relations and is one-to-one, 
we have the group net bundle $\efG = (G,i)_K$.
For each $\varrho \in \obj \wt \efT$ we define the family $\alpha_\varrho$ of unitary representations 
\[
\alpha_{\varrho,o} : G_o \to UH_{\varrho,o}
\ \ , \ \
u^o \mapsto u^\varrho_o
\ ; 
\]
since
\[
\{ \alpha_{\varrho,o'} \circ i_{o'o} (u^o) \} \circ R^\varrho_{o'o} =
\ad R^\varrho_{o'o}(u^o_\varrho) \circ R^\varrho_{o'o} =
R^\varrho_{o'o} \circ u^o_\varrho =
R^\varrho_{o'o} \circ \alpha_{\varrho,o}(u^o)
\ ,
\]
we have that (\ref{eq_GA00}) is fulfilled, so $\alpha$ is a gauge action of $\efG$ on $\efH_\varrho$. 
This means that $\efH_\varrho \in \obj {\bf bun}_\efG(K,{\bf Hilb})$.
At the level of arrows, if $t \in (\varrho,\varsigma)$ then $I(t) \in (\efH_\varrho,\efH_\varsigma)$ and $t_a \in {}_\beta(\varrho^a,\varsigma^a)$ for all $a \in K$; by definition of $G_a$ we have
\begin{equation}
\label{eq.Ta}
I_a ({}_\beta(\varrho^a,\varsigma^a)) = (H_{\varrho,a},H_{\varsigma,a})_{G_a}
\ \ , \ \
\forall a \in K
\ ,
\end{equation}
so
$I_a(t_a) \in (H_{\varrho,a},H_{\varsigma_a})_{G_a}$,
and $I(t) \in (\efH_\varrho,\efH_\varsigma)_\efG$.
Finally, if $T \in (\efH_\varrho,\efH_\varsigma)_\efG$ and $o \in K$
then by (\ref{eq.Ta}) we have $T_o = I_o(t_o)$ for a unique $t_o \in {}_\beta(\varrho^o,\varsigma^o)$; 
applying (\ref{eq_CCSA}) and using the fact that, by definition,
$T_{o'} = R^\varsigma_{o'o} \circ T_o \circ R^\varrho_{o o'}$, $\forall o \leq o'$, 
we find
\[
I_{o'}(t_{o'}) =
T_{o'} =
R^\varsigma_{o'o} \circ T_o \circ R^\varrho_{o o'} =
R^\varsigma_{o'o} \circ I_o(t_o) \circ R^\varrho_{o o'} =
I_{o'} \circ r_{o'o}(t_o)
\ .
\]
By injectivity of $I_{o'}$ we conclude that $t_{o'} = r_{o'o}(t_o)$ for all $o \leq o'$, 
thus $t := \{ t_o \}$ belongs to $(\varrho,\varsigma)$. We conclude that
$(\efH_\varrho,\efH_\varsigma)_\efG = \wt I(\varrho,\varsigma)$ 
and this proves the theorem.
\end{proof}

The following result gives a necessary condition to the existence of embeddings.
\begin{thm}
\label{cor.lift}
If a DR-presheaf $\efT$ has an embedding then every $\varrho \in \obj \wt \efT$ is liftable.
\end{thm}

\begin{proof}
Let $\varrho \in \obj \wt \efT$. 
We define $Z_\gamma$, $\gamma \in K(a)$, as in (\ref{def_zgamma}) and consider the holonomy of $\efH_\varrho$,
\begin{equation}
\label{eq_zrho}
z_\varrho (\gamma) :=
R^\varrho_{a,|b_n|} \circ R^\varrho_{|b_n|,\bo_n}
\cdots
R^\varrho_{\bo_1,|b_1|} \circ R^\varrho_{|b_1|,a}
\ \ , \ \
\gamma := b_n * \cdots * b_1 \in K(a)
\ .
\end{equation}
Applying repeatedly (\ref{eq_CCSA}) we find 
\[
\{ I_a \circ Z_\gamma \}(f) = \wa z_\varrho (\gamma) \circ I_a(f)
\ \ , \ \
\forall f \in {}_\beta(\varrho^{a,r},\varsigma^{a,s})
\  ,\
r,s \in \bN
\ ,
\]
where $\wa z_\varrho(\gamma)$ is the automorphism of the category of tensor powers of $H_{\varrho,a}$ defined as in (\ref{eq.NGa}). 
Comparing with (\ref{eq_CCS01}) and recalling (\ref{eq_CCS01a}) we conclude that 
\[
\wa z_\varrho (\gamma) \, | \, { \wa \pi_{G_\varrho} } \ = \ \wa \chi_\varrho(\gamma)
\ \ , \ \ 
\forall \gamma : a \to a
\ ,
\]
and by (\ref{eq.NG}) this happens if, and only if,
$z_\varrho$ takes values in $NG_\varrho$ and $\chi_\varrho = q \circ z_\varrho$.
\end{proof}

A comment to the previous theorem.
As we know, any $\varrho \in \obj \wt \efT$ defines the symmetric tensor net subbundle 
${}_\beta \wa \varrho$ of ${}_\beta \efT_{\otimes,\vareps}$.
We have seen in the proof of Theorem \ref{thm.hol} that any fibre ${}_\beta \wa \varrho_a$, $a \in K$, 
is isomorphic to the category $\wa \pi_{G_\varrho}$ of tensor powers of the
defining representation $\pi_{G_\varrho} : G_\varrho \to \ud$,
so by Theorem \ref{thm_OG} and Theorem \ref{thm_FL08} there is an embedding 
$I_\varrho : {}_\beta \wa \varrho_{\otimes,\vareps} \to \efC_{\otimes,\vartheta}$
if, and only if, $\varrho$ is liftable. 
Thus, the examples of \S \ref{sec_end} show that in general we cannot expect that 
existence and uniqueness of $I$ hold.

\section{The gerbe perspective.}
\label{sec.concl}

In this section we explain what we mean by a gerbe over a poset,
and show how this enters in the game in the scenario of section categories.
Given the group $G$ and the poset $K$, a \emph{$G$-gerbe} over $K$ is a pair
$\check \efG = (i,\delta)_K$,
such that
\begin{equation}
\label{eq.CON1}
i : \Sigma_1(K) \to {\bf aut}G
\ \ , \ \
\delta : \Sigma_2(K) \to G
\ \ : \ \
\ad \delta_c \circ i_{\partial_1c} = i_{\partial_0c} \circ i_{\partial_2c}
\ , \
\forall c \in \Sigma_2(K)
\ .
\end{equation}
The map $\delta$ encodes the obstruction to constructing the net bundle
$\efG := ( \{ Y_a \equiv G \} , i )_K$, that is well-defined when $\delta_c \equiv 1$.
In precise terms, $(i,\delta)_K$ is a non-abelian 1-cocycle with 
values in the crossed module $G \to {\bf aut}G$ (see \cite[\S 4]{Vas09} and related references).

In this paper we already have encountered group gerbes. If $\varrho$ is a section of 
the DR-presheaf $\efT$ then, applying the adjoint action to the cochain 
${\bf u}_\varrho : \Sigma_1(K) \to NG_\varrho \subseteq \ud$
of Theorem \ref{thm.gerbe}, we find that (\ref{eq.CON1}) is fulfilled by
\begin{equation}
\label{eq.CON.ex}
i_b := \ad {\bf u}_\varrho(b) |_{G_\varrho}
\ \ , \ \
\delta_c := d{\bf u}_\varrho(c)
\ \ , \ \
\forall b \in \Sigma_1(K)
\ , \
c \in \Sigma_2(K)
\ .
\end{equation}
Let us now define the sets
\[
N_1(K) := \{ b = \{ b_0 \leq |b| \in K \} \}
\ \ , \ \
N_2(K) := \{ c = \{ c_0 \leq c_1 \leq |c| \in K \} \} 
\ ;
\]
as remarked in \cite[\S 2.2]{RRV07}, there are inclusions 
$N_1(K) \subseteq \Sigma_1(K)$, $N_2(K) \subseteq \Sigma_2(K)$.
A \emph{$\check \efG$-gerbe of \sC algebras} is a triple 
$\check \efF = (F,\jmath,\alpha)_K$, 
where $F = \{ F_o \}$ is a family of \sC algebras with actions
    $\alpha_o : G \to {\bf aut}F_o$, $o \in K$,
and $\jmath = \{ \jmath_{o'o} : F_o \to F_{o'} \}_{o \leq o'}$ 
    is a family of *-monomorphisms such that, for any 
    $b \in N_1(K)$, $c \in N_2(K)$, $g \in G$,
    \begin{equation}
    \label{eq.CON2}
    \jmath_{|c| c_1} \circ \jmath_{c_1 c_0} = \alpha_{|c|}(\delta_c) \circ \jmath_{|c| c_0}
    \ \ , \ \
    \jmath_{|b| b_0} \circ \alpha_{b_0}(g)  = \alpha_{|b|}(i_b(g)) \circ \jmath_{|b|b_0}
    \ .
    \end{equation}
    The first of (\ref{eq.CON2}) generalizes in a natural way the precosheaf relations (\ref{eq.00}),
    whilst the second one is analogous to (\ref{eq_GA00}). 
    The previous equalities imply that the fixed point family 
    $\efA = (A,\jmath)_K$,
    $A_a := F_a^G$, $a \in K$,
    is a \emph{net} of \sC algebras. 
    When $\delta \equiv 1$, we have that $\check \efG$ collapses to a group net bundle, written $\efG$, 
    making $\check \efF$ a $\efG$-\emph{net} of \sC algebras.

We now illustrate the construction motivating the above notions. 
Take the $G_\varrho$-gerbe $\check \efG_\varrho = (i,\delta)_K$ defined by (\ref{eq.CON.ex}).
Then consider the constant family of \sC algebras $F := \{ F_o := \mO_d  \}$ and set
\begin{equation}
\label{eq.CON3}
\jmath_{|b|b_0} := \wa {\bf u}_\varrho(b) : F_{b_0} \to F_{|b|}
\ \ , \ \
\alpha_o(g) := \wa g \in {\bf aut}F_o
\ \ , \ \
\forall b \in N_1(K) \ , \ o \in K \ , \ g \in G_\varrho
\ ,
\end{equation}
where $\wa {\bf u}_\varrho(b)$ is defined by applying (\ref{eq_cuntz_dr}) 
to ${\bf u}_\varrho(b) \in NG_\varrho \subseteq \ud$ and 
$\wa g \in {\bf aut}\mO_d$ is given by (\ref{eq_FL04}).
As we shall see in \cite{VasX}, it turns out that 
$\check \efF_\varrho = (F,\jmath,\alpha)_K$ 
is a $\check \efG_\varrho$-gerbe of \sC algebras, whose fixed-point net is isomorphic 
to the net $\efA_\varrho$ of Remark \ref{rem_CCS00}.

\

\noindent \emph{Acknowledgements.} The author would like to thank J.E. Roberts for his interest
for (an old version of) this preprint, and G. Ruzzi for fruitful discussions.

\appendix

\section{Appendix: Locally constant bundles.}
\label{sec_lc}

In this section we describe the equivalence between the category of locally constant bundles on the manifold $M$ and that of net bundles over a good base $\Delta$ of $M$.

\

\noindent {\bf Bundles.} 
Given the topological category ${\bf C}$, a ${\bf C}$-\emph{bundle} is given by a continuous surjective map $p : \mB \to M$ such that each \emph{fibre} $B_x := p^{-1}(x)$, $x \in M$, is an object of ${\bf C}$. A \emph{morphism} $T \in (\mB,\mB')$ is given by a continuous map $f : \mB \to \mB'$ such that: (1) $p' \circ f = p$ (this implies that $f$ restricts to maps $f_x : B_x \to B'_x$, $x \in M$); (2) Each $f_x$, $x \in M$, is an arrow in $(B_x,B'_x)$.

In this way we have the category 
${\bf bun}(M,{\bf C})$,
with objects bundles on $M$ with fibres in $\obj {\bf C}$ and arrows the above defined morphisms.

The \emph{constant bundle} is given by the projection $p : M \times X \to M$, $X \in \obj {\bf C}$. Let $U \subset M$; we define the \emph{restriction} $p_U : \mB_U := p^{-1}(U) \to U$, $\mB_U \in {\bf bun}(U,{\bf C})$, and say that $\mB$ is \emph{locally trivial} whenever for each $x \in M$ there is a neighborhood $U \ni x$ such that there is an isomorphism $\alpha \in ( \mB_U , U \times X )$. We denote the set of isomorphism classes of locally trivial bundles with fibre $X$ ($X$-\emph{bundles}, to be concise) by
$\unl{\bf bun}(M,X)$, $X \in \obj {\bf C}$.

\begin{ex}{\it 
When ${\bf C} = {\bf Hilb}$ the above construction yields Hermitian vector bundles. 
When ${\bf C} = \Calg$ we obtain locally trivial bundles of \sC algebras, 
whose algebras of sections are continuous fields in the sense of \cite[Chap. X]{Dix}. 
When ${\bf C} = {\bf TopGr}$ we have group bundles (see \cite[\S 4]{Hus}).
} \end{ex}

\noindent {\bf Transition maps and locally constant bundles.} 
Now, ${\bf aut}X$ is a topological group; if we consider a cover $\{ Y_i \}$ with local charts $p_i : p_i^{-1}(Y_i) \to Y_i \times X$, then defining 
$( x , u_{ij,x}(v) ) := p_i \circ p_j^{-1}(x,v)$, $x \in Y_{ij} := Y_i \cap Y_j$, $v \in X$, 
yields \emph{transition maps} $u_{ij} : Y_{ij} \to {\bf aut}X$ satisfying the \emph{cocycle relations}
\[
u_{ij}(x) \cdot u_{jk}(x) \ = \ u_{ik}(x)
\ \ , \ \
\forall x \in Y_i \cap Y_j \cap Y_k
\ .
\]
A \emph{locally constant bundle} is a pair $(\mB,Y)$, where $\mB$ is an $X$-bundle and $Y := \{ Y_i \}$ 
is an open cover of $M$ with local charts $p_i : p_i^{-1}(Y_i) \to Y_i \times X$ such that 
the associated transition maps $u_{ij} : Y_{ij} \to {\bf aut}X$ are \emph{locally constant}. 
Given the locally constant bundle $(\mB',Y')$, a morphism $f \in (\mB,\mB')$ is said to be 
\emph{locally constant} whenever any map
\[
f_{il} : Y_i \cap Y'_l \to (X,X')
\ \ , \ \
f_{il}(x) := p'_l \circ f \circ p_i^{-1}(x)
\ ,
\]
is locally constant. This yields a \emph{non-full} subcategory of ${\bf bun}(M,{\bf C})$, denoted by
${\bf lc}(M,{\bf C})$.
Sometimes we will not mention the locally constant structure and will write $\mB$ instead of $(\mB,Y)$.

We stress that locally constant bundles $(\mB,Y),(\mB',Y')$ may be isomorphic in ${\bf bun}(M,{\bf C})$ but not in ${\bf lc}(M,{\bf C})$. 
Given $X \in \obj {\bf C}$, we denote the set of isomorphism classes of locally constant $X$-bundles by 
$\unl{\bf lc}(M,X)$.

\

\noindent {\bf The equivalence between locally constant bundles on $M$ and net bundles over $\Delta$} 
has been proved in \cite[Prop.33]{RRV07} using a cohomological language. The same result can proved
using the notion of holonomy: if $\mB \to M$ is a locally constant bundle
with fibre $X$, then the monodromy $\chi_\gamma \in {\bf aut}X$ of a loop $\gamma : [0,1] \to M$ is homotopy invariant 
(this does not happen for the holonomy of generic bundles), so it induces a representation
$\chi : \pi_1(M) \to {\bf aut}X$.
This yields an equivalence
\begin{equation}
\label{eq.p1.lc}
{\bf lc}(M,{\bf C}) \to {\bf hom}(\pi_1(M),{\bf C}) \ ,
\end{equation}
whose inverse (up to isomorphism) is given by the operation of assigning to 
$\chi \in {\hom}( \pi_1(M) , {\bf aut}X )$, $X \in \obj {\bf C}$, 
the induced bundle $\mB := \wt M \times_\chi X$, where $\wt M$ is the universal cover of $M$
regarded as a right $\pi_1(M)$-space (see \cite[\S I.2]{Kob}).
%

\begin{ex}{\it 
\label{ex_fpb}
The case ${\bf C} = {\bf TopGr}$ yields the category of locally constant group bundles. 
In particular, as in (\ref{eq_PR}), for any compact Lie group $G$ we may consider the right translation action $G \to {\bf homeo}G$.
Locally constant bundles $\mG \in {\bf lc}(M,G)$ that admit a holonomy representation $\pi_1(M) \to G \simeq R(G)$ are called \textbf{locally constant principal bundles}. 
We denote the category of locally constant principal bundles by ${\bf lcPr}(M,G)$.
When $M$ is a manifold and $G$ is a Lie group, a principal $G$-bundle is locally constant if, and only if, it is flat, 
i.e., it admits a connection with vanishing curvature (see \cite[Prop.I.2.6]{Kob}).
In the case ${\bf C} = {\bf Hilb}$ we have locally constant Hilbert bundles, 
which are flat Hermitian bundles when the fibre is finite-dimensional (\cite[Chap.I]{Kob}). 
In this case it is easily proved that (\ref{eq.p1.lc}) preserves direct sums and tensor products.
} 
\end{ex}

Combining (\ref{eq.p1.lc}) and Theorem \ref{thm_NET00} we obtain:
\begin{thm}
\label{thm_NET01}
Let $M$ be a space, $\Delta$ denote a good base of $M$ and $G$ a topological group. 
Then for any category ${\bf C}$ there are equivalences
\begin{equation}
\label{eq01_thm_NET01}
{\bf bun}(\Delta,{\bf C}) \to {\bf lc}(M,{\bf C}) 
\ \ , \ \
\efB \mapsto (\mB,Y)
\ ,
\end{equation}
\begin{equation}
\label{eq02_thm_NET01}
{\bf Pr}(\Delta,G) \to {\bf lcPr}(M,G)
\ \ , \ \
\efP \mapsto (\mP,Y)
\ .
\end{equation}
\end{thm}


\end{document}